\documentclass[3p,10pt]{elsarticle}

\usepackage{amssymb,amsthm}
\usepackage{latexsym,amsmath,subfigure,color,siunitx,tikz,pgfplots}

\usepackage{lineno}

\journal{}

\newcommand{\eps}{\varepsilon}
\newcommand{\set}[1]{\left\{#1\right\}}
\newcommand{\abs}[1]{\left|#1\right|}
\newcommand{\p}{\partial}
\newcommand{\rd}{\mathrm{d}}
\newcommand{\ma}{\mathbf{a}}
\newcommand{\mb}{\mathbf{b}}
\newcommand{\mc}{\mathbf{c}}
\newcommand{\mr}{\mathbf{r}}

\newcommand{\mA}{\mathcal{A}}

\newcommand{\mE}{\mathbf{E}}
\newcommand{\mF}{\mathbf{F}}
\newcommand{\mG}{\mathbf{G}}
\newcommand{\mH}{\mathbf{H}}
\newcommand{\vn}{\boldsymbol{\nu}}
\newcommand{\vt}{\boldsymbol{\theta}}
\newcommand{\vv}{\boldsymbol{\vartheta}}

\DeclareMathOperator*{\inc}{inc}
\DeclareMathOperator*{\osm}{OSM}
\DeclareMathOperator*{\fsm}{FSM}
\DeclareMathOperator*{\msm}{MSM}

\DeclareMathOperator*{\scat}{scat}

\DeclareMathOperator*{\area}{area}

\theoremstyle{plain}
\newtheorem{theorem}{Theorem}[section]
\newtheorem{lemma}{Lemma}[section]
\newtheorem{corollary}{Corollary}[section]
\newtheorem{remark}{Remark}[section]

\theoremstyle{remark}
\newtheorem{example}{Example}[section]

\begin{document}

\begin{frontmatter}



\title{Real-time inversion of two-dimensional Fresnel experimental database using orthogonality sampling method with single and multiple sources: the case of transverse electric polarized waves}
\author[HU]{Junyong Eom}
\ead{eom@es.hokudai.ac.jp}
\address[HU]{Research Institute for Electronic Science, Hokkaido University, Sapporo, 001-0020, Japan}
\author[PNU]{Sangwoo Kang\corref{corKang}}
\ead{sangwoo.kang@pusan.ac.kr}
\address[PNU]{Graduate School of Data Science, Pusan National University, Busan, 46241, Korea}
\author[KMU]{Minyeob Lee}
\ead{lmmmy0319@kookmin.ac.kr}
\author[KMU]{Won-Kwang Park\corref{corPark}}
\ead{parkwk@kookmin.ac.kr}
\address[KMU]{Department of Mathematics, Kookmin University, Seoul, 02707, Korea}
\cortext[corKang]{Co-oorresponding author}
\cortext[corPark]{Corresponding author}

\begin{abstract}
This paper concerns an application of the orthogonality sampling method (OSM) for a real-time identification of small objects from two-dimensional Fresnel experimental dataset in transverse electric polarization. First, we apply the OSM with a single source by designing an indicator function based on the asymptotic expansion formula for the scattered field in the presence of small objects. We demonstrate that the indicator function can be expressed by an infinite series of Bessel functions of integer order of the first kind, the range of the signal receiver, and the location of the emitter. Based on this, we then investigate the applicability and limitations of the designed OSM. Specifically, we find that the imaging performance is strongly dependent on the source and the applied frequency. We then apply the OSM with multiple sources to improve imaging performance. Based on the identified structure of the OSM with a single source, we design an indicator function with multiple sources and demonstrate that it can be expressed by an infinite series of the Bessel function of integer order of the first kind, and we explain that objects can be identified uniquely using the designed OSM. Numerical simulation results obtained with the Fresnel experimental dataset demonstrate the advantages and disadvantages of the OSM with a single source and confirm that the designed OSM with multiple sources improves imaging performance.
\end{abstract}

\begin{keyword}
Orthogonality sampling method \sep limited-aperture inverse scattering problem \sep Bessel functions \sep Fresnel experimental dataset


\end{keyword}

\end{frontmatter}






\section{Introduction}\label{sec:1}
The development of an effective and stable object detection and imaging technique is an important research subject in mathematics, engineering, physics, and geophysics because it has significant and valuable applications in various fields, including nondestructive evaluation \cite{FMGD,YCYZ}, biomedical imaging \cite{A1,A2}, geophysics \cite{P5,Z2}, astronomy \cite{CB2,SST}, and optics \cite{L8,MB}. To address this problem, various iterative (or quantitative) and noniterative (or qualitative) techniques have been investigated.

The orthogonality sampling method (OSM) is a noniterative technique to identify the location of scatterers using one or more incident fields \cite{P1}. Throughout several studies, the OSM has emerged as a potential technique that offers an effective balance between computational simplicity and accuracy for the detection of small objects in the inverse scattering problem \cite{ACP,G1,HN,KCP1} and microwave imaging \cite{A5,ACKLPS,P-OSM1,P-OSM2}. The OSM has been applied successfully to retrieve unknown objects from the Fresnel experimental dataset \cite{BIPAC,LNST,P-OSM3}; however, most studies focused on applying the OSM in transverse magnetic (TM) polarization. To the best of our knowledge, no theoretical study of the OSM has been conducted to retrieve small objects in transverse electric (TE) polarization with Fresnel dataset.

Thus, in this paper, we consider the application of the OSM to identify the existence, location, and shape of an unknown small object from a two-dimensional Fresnel experimental dataset in TE polarization. First, we design a new indicator function with a single source based on the fact that the measured scattered field data can be represented by an asymptotic expansion formula. To explain the applicability and fundamental limitation of object localization, we demonstrate that this indicator function can be represented by an infinite series of Bessel functions of integer order of the first kind, the range of the signal receiver, and the emitter location. This allows us to determine that it is possible to identify small objects using the designed indicator function; however, its imaging performance is strongly dependent on the applied frequency of operation and the emitter location. To verify the theoretical result, we present the results of a numerical simulation and examine some phenomena.

Then, to improve the imaging performance, we apply the OSM with multiple sources. Here, we design an indicator function with multiple sources and extensively demonstrate that the indicator function can be expressed by an infinite series of the Bessel functions of integer order of the first kind. Based on the theoretical result, we demonstrate that the designed indicator function improves the imaging performance and that the object can be identified uniquely. We also present simulation results to support the theoretical results and briefly introduce the imaging results with multiple sources and frequencies.

The remainder of this paper is organized as follows. In Section \ref{sec:2}, we briefly introduce the direct scattering problem and the asymptotic expansion formula in the presence of small objects. In Section \ref{sec:3}, we present the indicator function designed with a single source, analyze its structure by establishing a relationship with an infinite series of Bessel functions of integer order of the first kind, the range of the signal receiver, and the emitter location, and we explain some of the properties of the imaging results. In Section \ref{sec:4}, numerical simulation results obtained with the Fresnel experimental dataset are presented to support the theoretical result. Then, in Section \ref{sec:5}, we present the indicator function designed with multiple sources, reveal its mathematical structure by establishing a relationship with an infinite series of Bessel functions of integer order of the first kind, and we explain some properties of the imaging results. In Section \ref{sec:6}, numerical simulation results with multiple sources and frequencies are presented. Finally, the paper is concluded in Section \ref{sec:7}, including suggestions for future work.

Finally, let us emphasize that the result obtained using the OSM does not guarantee the complete shape of the object. Fortunately, once the outline shape of the object is identified, it can be adopted as an initial guess, which can then be made more accurate by applying iterative schemes, e.g., the level set method \cite{RLD}, the contrast source inversion method \cite{BAB}, the Bayesian method \cite{BPV2}, or the distorted-wave Born approach \cite{TBLH}.

\section{Direct scattering problem and asymptotic expansion formula}\label{sec:2}
Assume that there exist two-dimensional small objects $D_s$, $s=1,2,\ldots,S$ in a homogeneous region $\Omega\subset\mathbb{R}^2$. Throughout this paper, we assume that all $D_s$ are small ball with radius $\alpha_s$ centered at the $\mr_s$ and well separated from each others, and $\Omega$ is a subset of the interior of an anechoic chamber. Here, the values of the background conductivity, permeability, and permittivity are $\sigma_0\approx0$, $\mu_0=4\pi\times \SI{e-7}{\henry/\meter}$, and $\eps_0=\SI{8.854e-12}{\farad/\meter}$, respectively, and we denote $k=\omega\sqrt{\eps_0\mu_0}$ as the background wavenumber. In addition, each $D_s$ and $\Omega$ are characterized by the magnetic permeability value at the given angular frequency $\omega=2\pi f$. We denote $\mu_s$ as the permeability value of $D_s$ and introduce the following piecewise constant:
\[\mu(\mr)=\left\{\begin{array}{ccl}
\mu_s&\text{for}&\mr\in D_s\\
\mu_0&\text{for}&\mr\in\Omega\backslash\overline{D}.
\end{array}\right.\]

Note that we adopt the emitter and receiver arrangement setting from the literature \cite{BS}. Specifically, we denote $\ma_m$ and $\mb_{m,n}$ as the locations of the $m$th emitter $\mA_m$, where $m=1,2,\ldots,M$, and the corresponding $n$th receiver $\mathcal{B}_{m,n}$, where $n=1,2,\ldots,N$, respectively. Here, $\ma_m$ and $\mb_{m,n}$ can be expressed as follows:
\[\ma_m=|\ma_m|(\cos\vartheta_m,\sin\vartheta_m)=|\ma_m|\vv_m\quad\text{with}\quad|\ma_m|\equiv A=\SI{0.72}{\meter},\quad\vartheta_m=\frac{2(m-1)\pi}{M}\]
and
\[\mb_{m,n}=|\mb_{m,n}|(\cos\theta_{m,n},\sin\theta_{m,n})=|\mb_{m,n}|\vt_{m,n}\quad\text{with}\quad|\mb_{m,n}|\equiv B=\SI{0.76}{\meter},\quad\theta_{m,n}=\vartheta_m+\frac{\pi}{3}+\frac{4(n-1)\pi}{3(N-1)},\]
respectively. Here, $\vv_m\in\mathbb{S}^1$, $\vt_{m,n}\in\mathbb{S}_m^1$, where $\mathbb{S}^1$ denotes the unit circle centered at the origin, and \[\mathbb{S}_m^1=\set{(\cos\theta,\sin\theta):\vartheta_m+\frac{\pi}{3}\leq\theta\leq\vartheta_m+\frac{5\pi}{3}}\subset\mathbb{S}^1.\]
For an illustration, refer to Figure \ref{Configuration_Fresnel}. Then, the incident field at the fixed-point source $\mathcal{A}_m$ can be expressed as
\[u_{\inc}(\mr,\ma_m)=-\frac{i}{4}H_0^{(1)}(k|\mr-\ma_m|):=G(\mr,\ma_m),\quad\mr\in\Omega,\]
where $H_0^{(1)}$ denotes the Hankel function of order zero of the first kind. Correspondingly, the time-harmonic total field $u(\mb_{m,n},\mr)$ measured at the $n$th receiver $\mb_{m,n}$ satisfies
\[\nabla\cdot\left(\frac{1}{\mu(\mr)}\nabla u(\mb_{m,n},\mr)\right)+\omega^2\eps_0u(\mb_{m,n},\mr)=0\quad\text{for}\quad\mr\in\Omega\]
with the transmission condition that the boundary $\p D_s$, $s=1,2,\ldots,S$ (see \cite{AK2} for instance). In addition, the time harmonic $e^{-i\omega t}$ is assumed in this case.

Let $u_{\scat}(\mb_{m,n},\mr)$ denote the scattered field corresponding to the incident field. Then, $u_{\scat}(\mb_{m,n},\mr)$ can be expressed by the double-layer potential with an unknown density function $\psi$:
\[u_{\scat}(\mb_{m,n},\mr)=u(\mb_{m,n},\ma_m)-u_{\inc}(\mr,\ma_m)=\int_D \frac{\p G(\mb_{m,n},\mr)}{\p\vn(\mr)}\psi(\mr,\ma_m)\rd\mr,\]
where $\vn(\mr)$ is the unit outward normal at $\mr$. We refer to \cite{CK} for a detailed description. Note that without priori information of $D$, the closed form of the density function $\psi(\mr,\ma_m)$ is unknown; thus, it is inappropriate to employ $u_{\scat}(\mb_{m,n},\mr)$ directly to design the indicator function of the OSM. Therefore, we employ the following asymptotic expansion formula to design and analyze the structure of the indicator function.

\begin{lemma}
For sufficiently large $f$, $u_{\scat}(\mb_{m,n},\ma_m)$ can be expressed as follows:
\begin{equation}\label{AsymptoticFormula}
u_{\scat}(\mb_{m,n},\ma_m)=\sum_{s=1}^{S}\alpha_s^2\pi\left(\frac{\mu_0}{\mu_s+\mu_0}\right)\nabla G(\mb_{m,n},\mr_s)\cdot\nabla G(\ma_m,\mr_s)+o(\alpha_s^2).
\end{equation}
\end{lemma}

\begin{figure}[h]
\begin{center}
\begin{tikzpicture}[scale=2]
\def\RT{0.72*2};
\def\RR{0.76*2};
\def\Edge{0.15*2};

\draw[black,dashed,-stealth] (0,0) -- ({\RT*cos(0)},{\RT*sin(0)});
\draw[black,dashed,-] (0,0) -- ({\RT*cos(30)},{\RT*sin(30)});
\draw[black,dashed,-] (0,0) -- ({\RR*cos(90)},{\RR*sin(90)});
\draw[black,dashed,-] (0,0) -- ({\RR*cos(330)},{\RR*sin(330)});

\draw[black,solid,-stealth] ({0.25*cos(30)},{0.25*sin(30)}) arc (30:90:0.25);
\node[black] at ({0.5*cos(70)},{0.45*sin(70)}) {$\displaystyle\frac{\pi}{3}$};

\draw[black,solid,-stealth] ({0.2*cos(90)},{0.2*sin(90)}) arc (90:330:0.2);
\node[black] at ({0.35*cos(180)},{0.4*sin(180)}) {$\displaystyle\frac{4\pi}{3}$};

\draw[black,solid,-stealth] ({0.35*cos(0)},{0.35*sin(0)}) arc (0:30:0.35);
\node[black] at ({0.5*cos(30)},{0.22*sin(30)}) {$~~\vartheta_m$};

\draw[green,fill=green] ({\RT*cos(30)},{\RT*sin(30)}) circle (0.05cm) node[above,black,yshift=2] {$\mathcal{A}_m$};
\draw[black,fill=black] ({\RT*cos(30)},{\RT*sin(30)}) circle (0.02cm);

\foreach \beta in {90,95,...,330}
{\draw[red,fill=red] ({\RR*cos(\beta)},{\RR*sin(\beta)}) circle (0.05cm);
\draw[black,fill=black] ({\RR*cos(\beta)},{\RR*sin(\beta)}) circle (0.02cm);}

\node[right,xshift=2] at (0,\RR) {$\mathcal{B}_{m,1}$};
\node[above,yshift=2] at ({\RR*cos(330)},{\RR*sin(330)}) {$\mathcal{B}_{m,N}$};

\draw[green,fill=green] (-0.5,-0.5) circle (0.05cm);
\draw[black,fill=black] (-0.5,-0.5) circle (0.02cm);
\node[right] at (-0.5,-0.5) {\texttt{~emitter}};
\draw[red,fill=red] (-0.5,-0.7) circle (0.05cm);
\draw[black,fill=black] (-0.5,-0.7) circle (0.02cm);
\node[right] at (-0.5,-0.7) {\texttt{~receiver}};
\end{tikzpicture}
\quad
\begin{tikzpicture}[scale=2]
\def\RT{0.72*2};
\def\RR{0.76*2};
\def\Edge{0.15*2};

\draw[green!20!white,fill=green!20!white] ({\RT*cos(30)},{\RT*sin(30)}) circle (0.05cm) node[above,black,yshift=2] {$\mathcal{A}_m$};
\draw[black!20!white,fill=black!20!white] ({\RT*cos(30)},{\RT*sin(30)}) circle (0.02cm);

\foreach \beta in {90,95,...,330}
{\draw[red!20!white,fill=red!20!white] ({\RR*cos(\beta)},{\RR*sin(\beta)}) circle (0.05cm);
\draw[black!20!white,fill=black!20!white] ({\RR*cos(\beta)},{\RR*sin(\beta)}) circle (0.02cm);}

\node[right,xshift=2] at (0,\RR) {$\mathcal{B}_{m,1}$};
\node[above,yshift=2] at ({\RR*cos(330)},{\RR*sin(330)}) {$\mathcal{B}_{m,N}$};

\draw[cyan,solid,line width=2pt,-] ({\RR*cos(90)},{\RR*sin(90)}) arc (90:330:\RR);

\draw[cyan,solid,line width=2pt] (-0.5,-0.7) -- (-0.05,-0.7);
\node[right] at (-0.0,-0.7) {$\mathbb{S}_m^1$};
\end{tikzpicture}
\caption{\label{Configuration_Fresnel}Receiver arrangements corresponding to the emitter location.}
\end{center}
\end{figure}
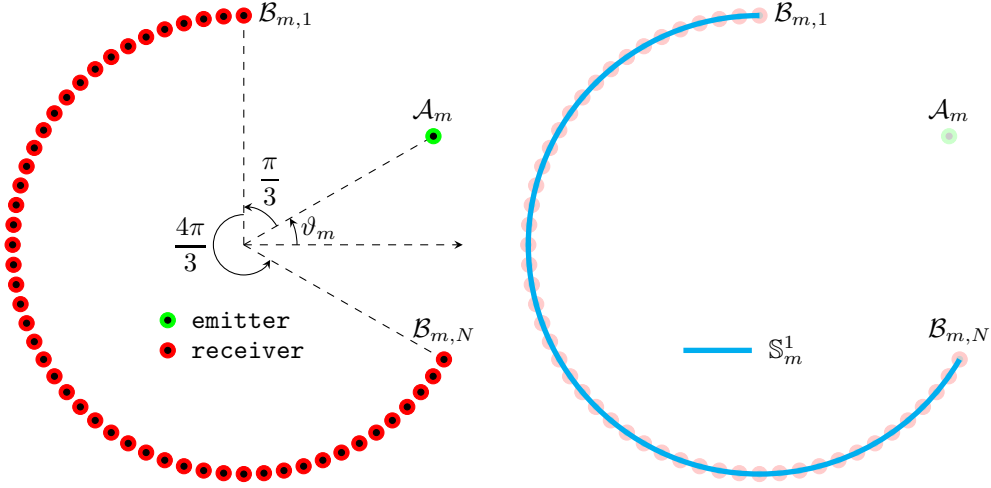
\section{Indicator function with a single source}\label{sec:3}
In the following, we consider the design of an indicator function with a single source. Here, let $\mE(m)$ be the following arrangement of the measured scattered field data with emitter $\mathcal{A}_m$:
\begin{equation}\label{ArrangementE}
\mE(m)=\Big(u_{\scat}(\mb_{m,1},\ma_m),u_{\scat}(\mb_{m,2},\ma_m),\ldots,u_{\scat}(\mb_{m,N},\ma_m)\Big).
\end{equation}
Then, applying \eqref{AsymptoticFormula} to \eqref{ArrangementE} yields
\[\mE(m)=\sum_{s=1}^{S}k^2\area(D_s)\left(\frac{\eps_s-\eps_0}{\eps_0\mu_0}\right)G(\mr_s,\ma_m)
\begin{pmatrix}
\nabla G(\mb_{m,1},\mr_s)\cdot\nabla G(\ma_m,\mr_s)\\
\nabla G(\mb_{m,2},\mr_s)\cdot\nabla G(\ma_m,\mr_s)\\
\vdots\\
\nabla G(\mb_{m,N},\mr_s)\cdot\nabla G(\ma_m,\mr_s)
\end{pmatrix}^T.\]

Thus, it appears natural to test the orthogonality relation between $u_{\scat}(\mb_{m,n},\ma_m)$ and $\nabla G(\mb_{m,N},\cdot)\cdot\nabla G(\ma_m,\cdot)$. Therefore, we introduce the following test vector: for each $\mr\in\Omega$,
\begin{equation}\label{TestVector1}
\mG(\mr)=\Big(\nabla G(\mb_{m,1},\mr)\cdot\nabla G(\ma_m,\mr),\nabla G(\mb_{m,2},\mr)\cdot\nabla G(\ma_m,\mr),\ldots,\nabla G(\mb_{m,N},\mr)\cdot\nabla G(\ma_m,\mr)\Big)
\end{equation}
and the corresponding indicator function
\begin{equation}\label{ImagingFunction_Single}
\mathfrak{F}_{\osm}(\mr,m)=|\mE(m)\cdot\overline{\mG(\mr)}|=\left|\sum_{n=1}^{N}u_{\scat}(\mb_{m,n},\ma_m)\overline{\nabla G(\mb_{m,n},\mr)\cdot\nabla G(\ma_m,\mr)}\right|.
\end{equation}
Then, the map of $\mathfrak{F}_{\osm}(\mr,m)$ will contain large magnitude peaks at $\mr=\mr_s\in D_s$; thus, it will be possible to recognize the existence or outline shape of $D_s$, $s=1,2,\ldots,S$.

\begin{remark}[Test vectors]
Rather than using \eqref{TestVector1}, different test vectors were adopted to design indicator functions of the OSM. For example, in the literature \cite{EP}, the following test vector, which was motivated by the imaging in TM polarization, was considered:
\[\mF(\mr)=\Big(G(\mb_{m,1},\mr),G(\mb_{m,2},\mr),\ldots,G(\mb_{m,N},\mr)\Big).\]
In addition, in another study \cite{AIL1}, by regarding $\nabla G(\ma_m,\mr)$ of \eqref{TestVector1} as a nonzero vector $\mc$, the following vector was adopted:
\[\mH(\mr)=\Big(\mc\cdot\nabla G(\mb_{m,1},\mr),\mc\cdot\nabla G(\mb_{m,2},\mr),\ldots,\mc\cdot\nabla G(\mb_{m,N},\mr)\Big).\]
\end{remark}

To explain the feasibility and properties of the $\mathfrak{F}_{\osm}(\mr,m)$, we explore its mathematical structure. To facilitate an effective exploration, we introduce the following relationships derived in the literature \cite{P-SUB17}. For reader's sake, we provide a brief derivation.

\begin{lemma}\label{Lemma_Integration}
Let $\mr=|\mr|(\cos\phi,\sin\phi)\in\mathbb{R}^2$ with $\vt\cdot\mr=|\mr|\cos(\theta-\phi)$ and $\vv\in\mathbb{S}^1$ with $\vt\cdot\vv=\cos(\theta-\vartheta)$. Then, the following relation holds uniformly
\begin{align}
\begin{aligned}\label{EquationBessel1}
\int_{\mathbb{S}_m^1}(\vt\cdot\vv)^2e^{ik\vt\cdot\mr}d\vt&=\int_{\theta_1}^{\theta_N}\cos^2(\theta-\vartheta_m)e^{ix\cos(\theta-\phi)}d\theta\\
&=(\theta_N-\theta_1)\left(\frac{1}{2}J_0(x)-\cos(2\vartheta_m-2\phi)J_2(x)\right)+\frac12\sin(\theta_N-\theta_1)\cos(\theta_N+\theta_1-2\vartheta_m)J_0(x)\\
&+\sum_{q=1}^{\infty}\frac{2i^q}{q}\sin\frac{q(\theta_N-\theta_1)}{2}\cos\frac{q(\theta_N+\theta_1-2\phi)}{2}J_q(x)\\
&+\sum_{q=1}^{\infty}\frac{2i^q}{q+2}\sin\frac{(q+2)(\theta_N-\theta_1)}{2}\cos\frac{(q+2)(\theta_N+\theta_1)-4\vartheta_m-2q\phi}{2}J_q(x)\\
&+\sum_{q=1,q\ne2}^{\infty}\frac{2i^q}{q-2}\sin\frac{(q-2)(\theta_N-\theta_1)}{2}\cos\frac{(q-2)(\theta_N+\theta_1)+4\vartheta_m-2q\phi}{2}J_q(x),
\end{aligned}
\end{align}
where $J_q$ denotes the Bessel function of order $q$.
\end{lemma}
\begin{proof}
Based on the uniform convergence of the Jacobi-Anger expansion formula,
\begin{equation}\label{JacobiAnger}
e^{ix\cos\phi}=J_0(x)+2\sum_{q=1}^{\infty}i^qJ_q(x)\cos(q\phi),
\end{equation}
we can obtain
\begin{multline}\label{formula1}
\int_{\mathbb{S}_m^1}(\vt\cdot\vv_m)^2e^{ik\vt\cdot\mr}d\vt=\int_{\theta_1}^{\theta_N}\cos^2(\theta-\vartheta_m)\left(J_0(kr)+2\sum_{q=1}^{\infty}i^qJ_q(kr)\cos(q(\theta-\phi))\right)d\theta\\
=\int_{\theta_1}^{\theta_N}\cos^2(\theta-\vartheta_m)J_0(kr)d\theta+2\sum_{q=1}^{\infty}i^qJ_q(kr)\int_{\theta_1}^{\theta_N}\cos^2(\theta-\vartheta_m)\cos(q(\theta-\phi))d\theta.
\end{multline}
Since
\[\int\cos^2(\theta-\vartheta_m)d\theta=\frac{\theta-\vartheta_m}{2}+\frac{\sin(2\theta-2\vartheta_m)}{4}\quad\text{and}\quad\sin x-\sin y=2\sin\frac{x-y}{2}\cos\frac{x+y}{2}\]
we can obtain
\begin{equation}\label{termI1}
\int_{\theta_1}^{\theta_N}\cos^2(\theta-\vartheta_m)J_0(kr)d\theta=\frac{(\theta_N-\theta_1)}{2}J_0(x)+\frac12\sin(\theta_N-\theta_1)\cos(\theta_N+\theta_1-2\vartheta_m)J_0(x).
\end{equation}
Furthermore, since
\begin{multline*}
\int\cos^2(\theta-\vartheta_m)\cos(q(\theta-\phi))d\theta=\frac12\int(1+\cos(2\theta-2\vartheta_m))\cos(q\theta-q\phi)d\theta\\
=\frac{1}{2q}\sin(q\theta-q\phi)+\frac{\sin[(q+2)\theta-2\vartheta_m-q\phi]}{2(q+2)}+\left\{
\begin{array}{ccc}
\displaystyle\medskip\frac{\theta}{2}\cos(2\vartheta_m-2\phi)&\text{if}&q=2\\
\displaystyle\frac{\sin[(q-2)\theta+2\vartheta_m-q\phi]}{2(q-2)}&\text{if}&q\ne2,
\end{array}\right.
\end{multline*}
we have
\begin{align}
\begin{aligned}\label{termI2}
2&\sum_{q=1}^{\infty}i^qJ_q(kr)\int_{\theta_1}^{\theta_N}\cos^2(\theta-\vartheta_m)\cos(q(\theta-\phi))d\theta\\
=&-(\theta_N-\theta_1)\cos(2\vartheta_m-2\phi)J_2(x)+\sum_{q=1}^{\infty}\frac{2i^q}{q}\sin\left(\frac{s(\theta_N-\theta_1)}{2}\right)\cos\left(\frac{s(\theta_N+\theta_1-2\phi)}{2}\right)J_q(x)\\
&+\sum_{q=1}^{\infty}\frac{2i^q}{q+2}\sin\left(\frac{(q+2)(\theta_N-\theta_1)}{2}\right)\cos\left(\frac{(q+2)(\theta_N+\theta_1)-4\vartheta_m-2q\phi}{2}\right)J_q(x)\\
&+\sum_{q=1,q\ne2}^{\infty}\frac{2i^q}{q-2}\sin\left(\frac{(q-2)(\theta_N-\theta_1)}{2}\right)\cos\left(\frac{(q-2)(\theta_N+\theta_1)+4\vartheta_m-2q\phi}{2}\right)J_q(x).
\end{aligned}
\end{align}
Combining (\ref{termI1}) and (\ref{termI2}), we can obtain the \eqref{EquationBessel1}.
\end{proof}

\begin{theorem}\label{OSM_Single}
  Let $\vv_m=(\cos\vartheta_m,\sin\vartheta_m)$, $\vt_{m,n}=(\cos\theta_{m,n},\sin\theta_{m,n})$, $\vt=(\cos\theta,\sin\theta)$, and $\mr-\mr_s=|\mr-\mr_s|(\cos\phi_s,\sin\phi_s)$. Here, if $4k|\mr-\mb_{m,n}|\gg1$ and $4k|\mr-\ma_m|\gg1$ for all $n=1,2,\ldots,N$, then $\mathfrak{F}_{\osm}(\mr,m)$ can be expressed as follows:
\begin{multline}\label{Structure_Single}
\mathfrak{F}_{\osm}(\mr,m)=\left|\frac{Nk^2}{4AB}\sum_{s=1}^{S}\alpha_s^2\left(\frac{\mu_0}{\mu_s+\mu_0}\right)e^{ik\vv_m\cdot(\mr-\mr_s)}\right.\\
\left.\times\left[\left(\frac12-\frac{3\sqrt{3}}{16\pi}\right)J_0(k|\mr-\mr_s|)-\cos(2\vartheta_m-2\phi_s)J_2(k|\mr-\mr_s|)+\mathcal{E}_{\osm}(\mr,m)\right]\right|,
\end{multline}
where 
  \begin{align*}
  \mathcal{E}_{\osm}(\mr,m)&=\frac{3}{2\pi}\sum_{q=1}^{\infty}\frac{i^q}{q}\sin\frac{2q\pi}{3}\cos(q\vartheta_m-q\phi_s)J_q(k|\mr-\mr_s|)\\
&+\frac{3}{2\pi}\sum_{q=1}^{\infty}\frac{(-i)^q}{q+2}\sin\frac{2(q+2)\pi}{3}\cos(q\vartheta_m-q\phi_s)J_q(k|\mr-\mr_s|)\\
&+\frac{3}{2\pi}\sum_{q=1,q\ne2}^{\infty}\frac{(-i)^q}{q-2}\sin\frac{2(q-2)\pi}{3}\cos(q\vartheta_m-q\phi_s)J_q(k|\mr-\mr_s|).
  \end{align*}
\end{theorem}
\begin{proof}
Since $4k|\mr-\mb_{m,n}|\gg1$ for $n=1,2,\cdots,N$, the following asymptotic forms hold (refer to the literature \cite{CK} for additional information).
\begin{align}
\begin{aligned}\label{Asymptotic_Hankel}
H_0^{(1)}(k|\mb_{m,n}-\mr|)&\approx\frac{(1-i)e^{ik|\mb_{m,n}|}}{\sqrt{k|\mb_{m,n}|\pi}}e^{-ik\vt_{m,n}\cdot\mr}\\
\nabla H_0^{(1)}(k|\mb_{m,n}-\mr|)&\approx-\frac{k(1+i)e^{ik|\mb_{m,n}|}}{\sqrt{k|\mb_{m,n}|\pi}}\vt_{m,n} e^{-ik\vt_{m,n}\cdot\mr}.
\end{aligned}
\end{align}
Thus, we can examine
\begin{align}
\begin{aligned}\label{Representation_exp}
&u_{\scat}(\mb_{m,n},\ma_m)\approx\frac{ik e^{ik(A+B)}}{8\sqrt{AB}}\sum_{s=1}^{S}\alpha_s^2\pi\left(\frac{\mu_0}{\mu_s+\mu_0}\right)(\vt_{m,n}\cdot\vv_m)e^{-ik(\vt_{m,n}+\vv_m)\cdot\mr_s}\\
&\nabla G(\mb_{m,1},\mr)\cdot\nabla G(\ma_m,\mr)\approx\frac{-2ik e^{ik(A+B)}}{\pi\sqrt{AB}}(\vt_{m,n}\cdot\vv_m)e^{-ik(\vt_{m,n}+\vv_m)\cdot\mr}.
\end{aligned}
\end{align}
Therefore, based on \eqref{Representation_exp}, we obtain the following:
\begin{align*}
&\sum_{n=1}^{N}u_{\scat}(\mb_{m,n},\ma_m)\overline{\nabla G(\mb_{m,N},\mr)\cdot\nabla G(\ma_m,\mr)}\\
&\approx-\sum_{n=1}^{N}\frac{k^2}{4AB\pi}\sum_{s=1}^{S}\alpha_s^2\pi\left(\frac{\mu_0}{\mu_s+\mu_0}\right)(\vt_{m,n}\cdot\vv_m)^2e^{ik(\vt_{m,n}+\vv_m)\cdot(\mr-\mr_s)}\\
&=-\frac{k^2}{4AB\pi}\sum_{s=1}^{S}\alpha_s^2\pi\left(\frac{\mu_0}{\mu_s+\mu_0}\right)e^{ik\vv_m\cdot(\mr-\mr_s)}\left(\sum_{n=1}^{N}(\vt_{m,n}\cdot\vv_m)^2e^{ik\vt_{m,n}\cdot(\mr-\mr_s)}\right).
\end{align*}

Because $N$ is sufficiently large and $\vt_{m,n}\cdot(\mr-\mr_s)=|\mr-\mr_s|\cos(\theta_{m,n}-\phi_s)$, we can derive the following by letting $\triangle\theta=(\theta_{m,N}-\theta_{m,1})/N=4\pi/3N$:
\begin{align*}
&\sum_{n=1}^{N}(\vt_{m,n}\cdot\vv_m)^2e^{ik\vt_{m,n}\cdot(\mr-\mr_s)}=\frac{N}{N\triangle\theta}\sum_{n=1}^{N}(\vt_{m,n}\cdot\vv_m)^2e^{ik\vt_{m,n}\cdot(\mr-\mr_s)}\triangle\theta\\
&\approx\frac{3N}{4\pi}\int_{\mathbb{S}_m^1}(\vt\cdot\vv_m)^2e^{ik\vt\cdot(\mr-\mr_s)}\rd\vt=\frac{3N}{4\pi}\int_{\vartheta_m+\pi/3}^{\vartheta_m+5\pi/3}\cos^2(\theta-\vartheta_m)e^{ik|\mr-\mr_s|\cos(\theta-\phi_s)}\rd\theta\\
&=N\left[\bigg(\frac{1}{2}-\frac{3\sqrt{3}}{16\pi}\bigg)J_0(k|\mr-\mr_s|)-\cos(2\vartheta_m-2\phi_s)J_2(k|\mr-\mr_s|)\right]\\
&+\frac{3N}{2\pi}\sum_{q=1}^{\infty}\frac{i^q}{q}\sin\frac{2q\pi}{3}\cos(q\vartheta_m-q\phi_s)J_q(k|\mr-\mr_s|)\\
&+\frac{3N}{2\pi}\left[\sum_{q=1}^{\infty}\frac{(-i)^q}{q+2}\sin\frac{2(q+2)\pi}{3}+\sum_{q=1,q\ne2}^{\infty}\frac{(-i)^q}{q-2}\sin\frac{2(q-2)\pi}{3}\right]\cos(q\vartheta_m-q\phi_s)J_q(k|\mr-\mr_s|).
\end{align*}
Hence, the structure \eqref{Structure_Single} can be obtained.
\end{proof}

Then, based on Theorem \ref{OSM_Single}, we can examine some properties of the OSM.

\begin{remark}[Dependence on the emitter location]\label{Remark_Single1}
Since $J_0(0)=1$ and $J_q(0)=0$ for nonzero integer $q$, the map of $\mathfrak{F}_{\osm}(\mr,m)$ will contain peaks of magnitudes
\[\frac{N\alpha_s^2k^2}{4AB}\left(\frac{\mu_0}{\mu_s+\mu_0}\right)\left(\frac12-\frac{3\sqrt{3}}{16\pi}\right)\]
at the location $\mr=\mr_s\in D_s$. Thus, it is likely possible to identify both the existence and location of $D_s$. Unfortunately, two large-magnitude peaks are also appear in the neighborhood of $\mr_s$. To visualize this phenomenon, we consider the following quantity
\[\mathcal{D}_{\osm}^{(1)}(x)=\left(\frac12-\frac{3\sqrt{3}}{16\pi}\right)J_0(k|x|)-J_2(k|x|).\]
This is similar to the value of $\mathfrak{F}_{\osm}(\mr,m)$ in the presence of single object $D_1$ located at the origin $\mr_1=(0,0)$ and $\vartheta_m$ is parallel to $\mr-\mr_1$. Based on the plot of $|\mathcal{D}_{\osm}^{(1)}(x)|$ in Figure \ref{PlotOSM1} at $f=2,6,\SI{10}{\giga\hertz}$, we can examine that the value $|\mathcal{D}_{\osm}^{(1)}(x)|$ has its local maximum at $x=0$ but its global maximum value is appear at $x=\pm0.3054/k$. Therefore, we can say that the location of $D_s$ can be recognized through the local maximum value of $\mathfrak{F}_{\osm}(\mr,m)$ and two peaks of large magnitude are also appeared at $\mr$ satisfying
\[k|\mr-\mr_s|=0.3054\quad\text{and}\quad\frac{\mr-\mr_s}{|\mr-\mr_s|}=\pm\vv_m.\]
Therefore, based on the above observations, the imaging performance of $\mathfrak{F}_{\osm}(\mr,m)$ will be strongly dependent on the position of the emitter $\mA_m$.
\end{remark}

\begin{figure}[h]
\begin{center}
\begin{tikzpicture}
\scriptsize
\begin{axis}
[legend style={at={(1,1)}},
width=\textwidth,
height=0.4\textwidth,
xmin=-0.1,
xmax=0.1,
ymin=0,
ymax=0.625,
legend cell align={left}]
\addplot[line width=1pt,solid,color=green!60!black] %
	table[x=x,y=y1,col sep=comma]{PlotOSM1.csv};
\addlegendentry{\scriptsize$f=\SI{2}{\giga\hertz}$};
\addplot[line width=1pt,solid,color=red] %
	table[x=x,y=y2,col sep=comma]{PlotOSM1.csv};
\addlegendentry{\scriptsize$f=\SI{6}{\giga\hertz}$};
\addplot[line width=1pt,solid,color=cyan!80!black] %
	table[x=x,y=y3,col sep=comma]{PlotOSM1.csv};
\addlegendentry{\scriptsize$f=\SI{10}{\giga\hertz}$};
\end{axis}
\end{tikzpicture}
\caption{\label{PlotOSM1}Plots of $|\mathcal{D}_{\osm}^{(1)}(x)|$ for $-0.1\leq x\leq0.1$ at $f=2,6,\SI{10}{\giga\hertz}$.}
\end{center}
\end{figure}

\begin{remark}[Structure of the factor $\mathcal{E}_{\osm}(\mr,m)$]\label{Remark_Single2}
Notice that since
\[\cos(\vartheta_m-\phi_s)=\frac{\mr-\mr_s}{|\mr-\mr_s|}\cdot\vv_m\quad\text{and}\quad\cos qx=\sum_{p=0}^{q}a_p^{(q)}\cos^px,\]
where coefficients $a_p^{(q)}$ satisfy the following relationships
\begin{align*}
\cos 2x &= 2 \cos^2 x - 1\\
\cos 3x &= 4 \cos^3 x - 3 \cos x\\
\cos 4x &= 8 \cos^4 x - 8 \cos^2 x + 1\\
\cos 5x &= 16\cos^5 x - 20 \cos^3 x + 5 \cos x\\
\cos 6x &= 32\cos^6 x - 48 \cos^4 x + 18\cos^2 x - 1\\
\cos 7x &= 64\cos^7 x - 112 \cos^5 x + 56\cos^3 x - 7 \cos x\\
&\vdots
\end{align*}
The structure of $\mathcal{E}_{\osm}(\mr,m)$ is as follows
\begin{align*}
\mathcal{E}_{\osm}(\mr,m)&=\frac{3}{2\pi}\sum_{q=1}^{\infty}\frac{i^q}{q}\sin\frac{2q\pi}{3}a_p^{(q)}\left(\frac{\mr-\mr_s}{|\mr-\mr_s|}\cdot\vv_m\right)^qJ_q(k|\mr-\mr_s|)\\
&+\frac{3}{2\pi}\sum_{q=1}^{\infty}\frac{(-i)^q}{q+2}\sin\frac{2(q+2)\pi}{3}a_p^{(q)}\left(\frac{\mr-\mr_s}{|\mr-\mr_s|}\cdot\vv_m\right)^qJ_q(k|\mr-\mr_s|)\\
&+\frac{3}{2\pi}\sum_{q=1,q\ne2}^{\infty}\frac{(-i)^q}{q-2}\sin\frac{2(q-2)\pi}{3}a_p^{(q)}\left(\frac{\mr-\mr_s}{|\mr-\mr_s|}\cdot\vv_m\right)^qJ_q(k|\mr-\mr_s|).
\end{align*}
Hence, we can say that the factor $\mathcal{E}_{\osm}(\mr,m)$ does not contribute to identify the objects and contribute to generate two peaks of large magnitude in the neighborhood of the $\mr_s$, $s=1,2,\ldots,S$.
\end{remark}

\begin{remark}[Influence of the factor $\mathcal{E}_{\osm}(\mr,m)$]\label{Remark_Single3}
The factor $\mathcal{E}_{\osm}(\mr,m)$ will disturb the identification process by generating several artifacts due to the oscillating property of $J_q$. To examine the influence of $\mathcal{E}_{\osm}(\mr,m)$, we consider the following quantity:
\[\mathcal{D}_{\osm}^{(2)}(x)=\frac{3}{2\pi}\sum_{q=1}^{10^5}\left(\frac{i^q}{q}\sin\frac{2q\pi}{3}+\frac{(-i)^q}{q+2}\sin\frac{2(q+2)\pi}{3}\right)J_q(k|x|)+\frac{3}{2\pi}\sum_{q=1,q\ne2}^{10^5}\frac{(-i)^q}{q-2}\sin\frac{2(q-2)\pi}{3}J_q(k|x|).\]
This is similar to the value of $\mathcal{E}_{\osm}(\mr,m)$ in the presence of single object $D_1$ located at the origin $\mr_1=(0,0)$ and $\vartheta_m$ is parallel to $\mr-\mr_1$ and $\vartheta_m=\phi_1$. Based on  numerical computation, $|\mathcal{D}_{\osm}^{(2)}(x)|\leq0.1744$ (see Figure \ref{PlotOSM2}), the value of $|\mathcal{E}_{\osm}(\mr,m)|$ will not exceed $0.18$. Therefore, the factor $\mathcal{E}_{\osm}(\mr,m)$ will contribute to the generation of several artifacts that disturb the recognition of the existence of objects.
\end{remark}

\begin{figure}[h]
\begin{center}
\begin{tikzpicture}
\scriptsize
\begin{axis}
[legend style={at={(1,1)}},
width=\textwidth,
height=0.4\textwidth,
ytick distance=0.18,
xmin=-0.1,
xmax=0.1,
ymin=0,
ymax=0.18,
legend cell align={left}]
\addplot[line width=1pt,solid,color=green!60!black] %
	table[x=x,y=y1,col sep=comma]{PlotOSM2.csv};
\addlegendentry{\scriptsize$f=\SI{2}{\giga\hertz}$};
\addplot[line width=1pt,solid,color=red] %
	table[x=x,y=y2,col sep=comma]{PlotOSM2.csv};
\addlegendentry{\scriptsize$f=\SI{6}{\giga\hertz}$};
\addplot[line width=1pt,solid,color=cyan!80!black] %
	table[x=x,y=y3,col sep=comma]{PlotOSM2.csv};
\addlegendentry{\scriptsize$f=\SI{10}{\giga\hertz}$};
\end{axis}
\end{tikzpicture}
\caption{\label{PlotOSM2}Plots of $|\mathcal{D}_{\osm}^{(2)}(x)|$ for $-0.1\leq x\leq0.1$ at $f=2,6,\SI{10}{\giga\hertz}$.}
\end{center}
\end{figure}
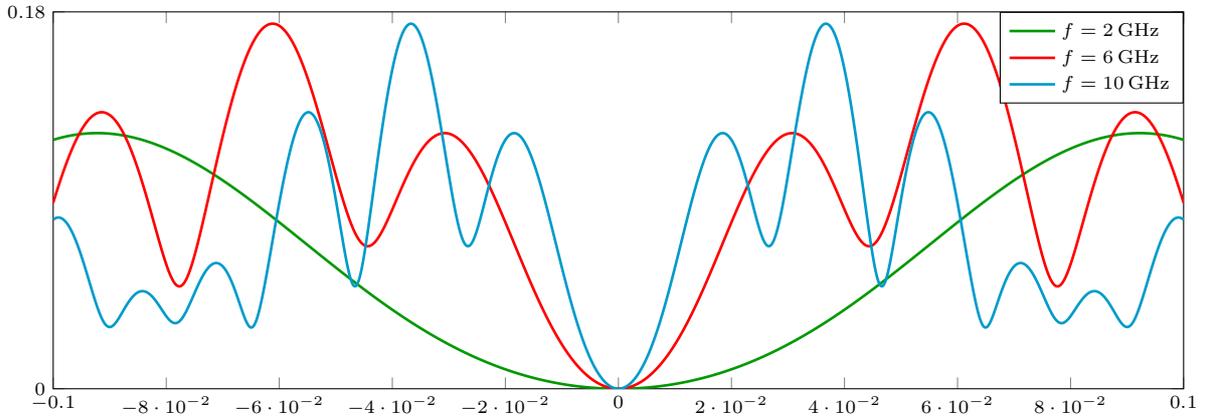

\begin{remark}[Limitation of object detection]\label{Remark_Single4}
Based on the Remark \ref{Remark_Single1}, the imaging performance is significantly dependent on the emitter location. Moreover, it will be very difficult to identify the object through the map of $\mathfrak{F}_{\osm}(\mr,m)$. To this end, we consider the following quantity:
\[\mathcal{D}_{\osm}(x)=\mathcal{D}_{\osm}^{(1)}(x)+\mathcal{D}_{\osm}^{(2)}(x),\]
where $\mathcal{D}_{\osm}^{(1)}(x)$ and $\mathcal{D}_{\osm}^{(2)}(x)$ are given in Remarks \ref{Remark_Single1} and \ref{Remark_Single3}, respectively. Based on numerical computation, it is impossible to say that the object can be identified uniquely through the OSM with single source because the value of $\mathcal{D}_{\osm}(x)$ does not reach its global maximum value at $x=0$ and several artifacts with large magnitude will disturb the object detection in the imaging results, refer to Figure \ref{PlotOSM}.
\end{remark}

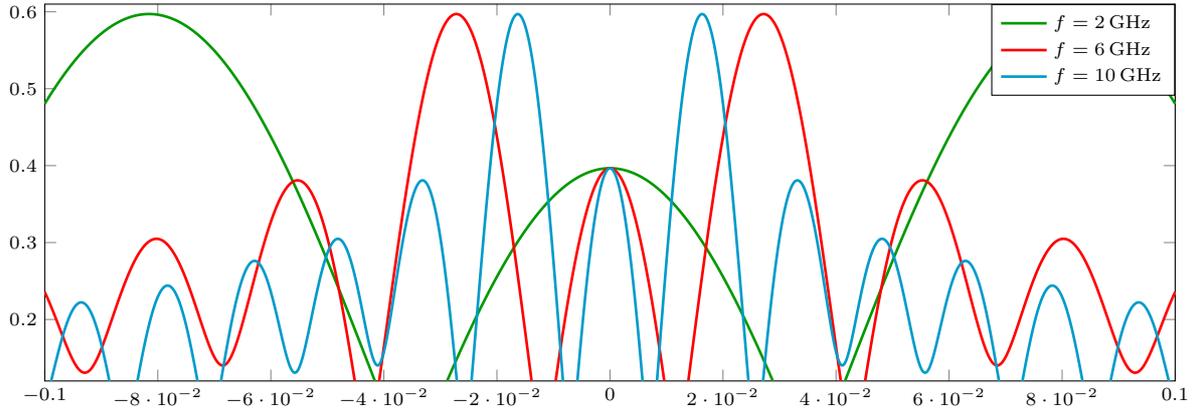
\begin{figure}[h]
\begin{center}
\begin{tikzpicture}
\scriptsize
\begin{axis}
[legend style={at={(1,1)}},
width=\textwidth,
height=0.4\textwidth,
xmin=-0.1,
xmax=0.1,
ymin=0.12,
ymax=0.61,
legend cell align={left}]
\addplot[line width=1pt,solid,color=green!60!black] %
	table[x=x,y=y1,col sep=comma]{PlotOSM.csv};
\addlegendentry{\scriptsize$f=\SI{2}{\giga\hertz}$};
\addplot[line width=1pt,solid,color=red] %
	table[x=x,y=y2,col sep=comma]{PlotOSM.csv};
\addlegendentry{\scriptsize$f=\SI{6}{\giga\hertz}$};
\addplot[line width=1pt,solid,color=cyan!80!black] %
	table[x=x,y=y3,col sep=comma]{PlotOSM.csv};
\addlegendentry{\scriptsize$f=\SI{10}{\giga\hertz}$};
\end{axis}
\end{tikzpicture}
\caption{\label{PlotOSM}Plots of $|\mathcal{D}_{\osm}(x)|$ for $-0.1\leq x\leq0.1$ at $f=2,6,\SI{10}{\giga\hertz}$.}
\end{center}
\end{figure}

\section{Simulation results with single source using synthetic and experimental dataset}\label{sec:4}
In the following, we present numerical simulation results using synthetic and Fresnel experimental dataset to support Theorem \ref{OSM_Single}. As discussed in Section \ref{sec:2}, the emitters and receivers were placed on circles centered at the origin with radii $A=\SI{0.72}{\meter}$ and $B=\SI{0.76}{\meter}$, respectively, and the imaging region $\Omega$ was selected as a square $(-\SI{0.1}{\meter},\SI{0.1}{\meter})\times(-\SI{0.1}{\meter},\SI{0.1}{\meter})$ to satisfy the relation $4k|\mr-\ma_m|$, $4k|\mr-\mb_{m,n}|\gg1$ for $m=1,2,\ldots,M(=36)$ and $n=1,2,\ldots,N(=49)$. In addition, the range of the receivers was restricted from $\SI{60}{\degree}$ to $\SI{300}{\degree}$ with a step size of $\triangle\theta=\SI{5}{\degree}$ based on the location of each emitter. The simulation setup is illustrated in Figure \ref{Configuration_Fresnel}.

\subsection{Simulation results using synthetic dataset}
For the numerical simulation using synthetic dataset, we chosen three small circles $D_s$, $s=1,2,3$, with radius $\alpha_s$, permittivity $\eps_0$, permeability $\mu_s$, and locations $\mr_1 = (\SI{0.07}{\meter},\SI{0.05}{\meter})$, $\mr_2 = (-\SI{0.07}{\meter},\SI{0.00}{\meter})$, and $\mr_3 = (\SI{0.04}{\meter},-\SI{0.06}{\meter})$. Values of permeabilities $\mu_m$ and radii $\alpha_m$ are listed in Table \ref{Table}. With this configuration, the measurement data $u_{\scat}(\mb_{m,n},\ma_m)$ at $f=4,8,\SI{12}{\giga\hertz}$ was generated by the Foldy-Lax framework \cite{HSZ4}. After the generation of the dataset, $\SI{20}{\dB}$ white Gaussian random noise was added to the unperturbed data.

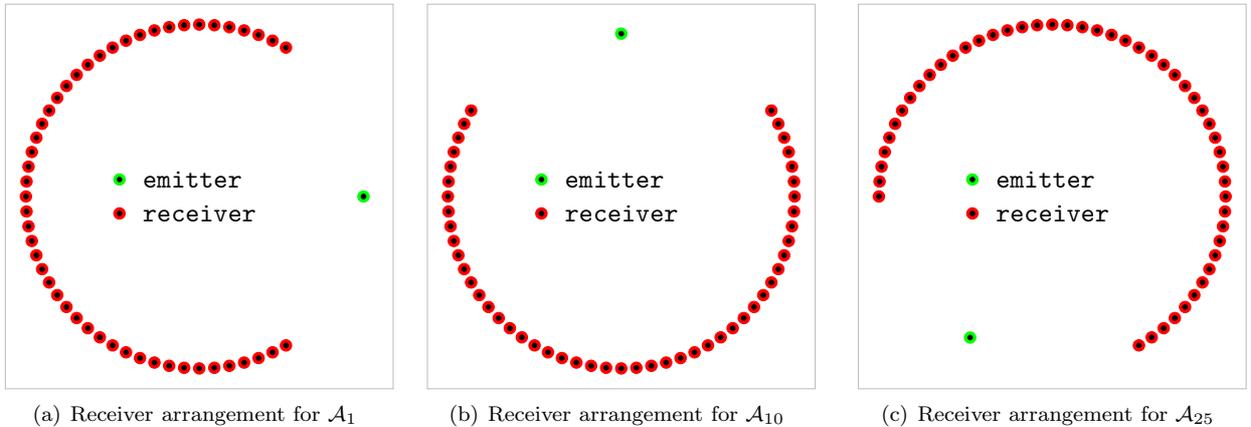
\begin{figure}[h]
\begin{center}
\subfigure[Receiver arrangement for $\mA_1$]{
\begin{tikzpicture}[scale=1.5]
\def\RT{0.72*2};
\def\RR{0.76*2};
\def\Edge{0.15*2};
\def\BD{0.85*2};
\draw[green,fill=green] ({\RT*cos(0)},{\RT*sin(0)}) circle (0.05cm);
\draw[black,fill=black] ({\RT*cos(0)},{\RT*sin(0)}) circle (0.02cm);
\foreach \beta in {60,65,...,300}
{\draw[red,fill=red] ({\RR*cos(\beta)},{\RR*sin(\beta)}) circle (0.05cm);
\draw[black,fill=black] ({\RR*cos(\beta)},{\RR*sin(\beta)}) circle (0.02cm);}

\draw[green,fill=green] (-0.7,0.15) circle (0.05cm);
\draw[black,fill=black] (-0.7,0.15) circle (0.02cm);
\node[right] at (-0.7,0.15) {\texttt{~emitter}};
\draw[red,fill=red] (-0.7,-0.15) circle (0.05cm);
\draw[black,fill=black] (-0.7,-0.15) circle (0.02cm);
\node[right] at (-0.7,-0.15) {\texttt{~receiver}};

\draw[gray!50!white] (\BD,\BD) -- (\BD,-\BD) -- (-\BD,-\BD) -- (-\BD,\BD) -- cycle;
\end{tikzpicture}}\hfill
\subfigure[Receiver arrangement for $\mA_{10}$]{
\begin{tikzpicture}[scale=1.5]

\def\RT{0.72*2};
\def\RR{0.76*2};
\def\Edge{0.15*2};
\def\BD{0.85*2};

\draw[green,fill=green] ({\RT*cos(90)},{\RT*sin(90)}) circle (0.05cm);
\draw[black,fill=black] ({\RT*cos(90)},{\RT*sin(90)}) circle (0.02cm);

\foreach \beta in {150,155,...,390}
{\draw[red,fill=red] ({\RR*cos(\beta)},{\RR*sin(\beta)}) circle (0.05cm);
\draw[black,fill=black] ({\RR*cos(\beta)},{\RR*sin(\beta)}) circle (0.02cm);}

\draw[green,fill=green] (-0.7,0.15) circle (0.05cm);
\draw[black,fill=black] (-0.7,0.15) circle (0.02cm);
\node[right] at (-0.7,0.15) {\texttt{~emitter}};
\draw[red,fill=red] (-0.7,-0.15) circle (0.05cm);
\draw[black,fill=black] (-0.7,-0.15) circle (0.02cm);
\node[right] at (-0.7,-0.15) {\texttt{~receiver}};

\draw[gray!50!white] (\BD,\BD) -- (\BD,-\BD) -- (-\BD,-\BD) -- (-\BD,\BD) -- cycle;
\end{tikzpicture}}
\hfill
\subfigure[Receiver arrangement for $\mA_{25}$]{
\begin{tikzpicture}[scale=1.5]

\def\RT{0.72*2};
\def\RR{0.76*2};
\def\Edge{0.15*2};
\def\BD{0.85*2};

\draw[green,fill=green] ({\RT*cos(240)},{\RT*sin(240)}) circle (0.05cm);
\draw[black,fill=black] ({\RT*cos(240)},{\RT*sin(240)}) circle (0.02cm);

\foreach \beta in {300,305,...,540}
{\draw[red,fill=red] ({\RR*cos(\beta)},{\RR*sin(\beta)}) circle (0.05cm);
\draw[black,fill=black] ({\RR*cos(\beta)},{\RR*sin(\beta)}) circle (0.02cm);}

\draw[green,fill=green] (-0.7,0.15) circle (0.05cm);
\draw[black,fill=black] (-0.7,0.15) circle (0.02cm);
\node[right] at (-0.7,0.15) {\texttt{~emitter}};
\draw[red,fill=red] (-0.7,-0.15) circle (0.05cm);
\draw[black,fill=black] (-0.7,-0.15) circle (0.02cm);
\node[right] at (-0.7,-0.15) {\texttt{~receiver}};

\draw[gray!50!white] (\BD,\BD) -- (\BD,-\BD) -- (-\BD,-\BD) -- (-\BD,\BD) -- cycle;
\end{tikzpicture}}
\caption{\label{SimulationSetting}Antenna arrangements with emitters $\mA_1$, $\mA_{10}$, and $\mA_{25}$.}
\end{center}
\end{figure}

\begin{table}[h]
\begin{center}
\begin{tabular}{llllllll} \hline
Cases&$\mu_1$&$\mu_2$&$\mu_3$&$\alpha_1$&$\alpha_2$&$\alpha_3$&description\\\hline
Case $1$&$5\mu_0$&$5\mu_0$&$5\mu_0$&$\SI{0.010}{\meter}$&$\SI{0.010}{\meter}$&$\SI{0.010}{\meter}$&same size and permittivity\\
Case $2$&$3\mu_0$&$5\mu_0$&$7\mu_0$&$\SI{0.010}{\meter}$&$\SI{0.010}{\meter}$&$\SI{0.010}{\meter}$&same size but different permeability\\
Case $3$&$5\mu_0$&$5\mu_0$&$5\mu_0$&$\SI{0.015}{\meter}$&$\SI{0.010}{\meter}$&$\SI{0.005}{\meter}$&same permeability but different size\\
Case $4$&$3\mu_0$&$5\mu_0$&$7\mu_0$&$\SI{0.015}{\meter}$&$\SI{0.010}{\meter}$&$\SI{0.005}{\meter}$&different size and permeability\\\hline
\end{tabular}
\caption{\label{Table}Permeabilities and sizes of objects in synthetic data experiment.}
\end{center}
\end{table}

\begin{example}\label{Ex_Single_Case1}
Figure \ref{Synthetic_Single_Case1} shows maps of $\mathfrak{F}_{\osm}(\mr,1)$ with $\vv_1=\SI{0.72}{\meter}(\cos\SI{0}{\degree},\sin\SI{0}{\degree})$ in Case $1$. Based on the imaging result, we can observe that although the existence of $D_s$, $s=1,2,3$, can be identified, their exact locations and outline shapes cannot be recognized at $f=\SI{4}{\giga\hertz}$. Notice that based on the imaging result at $f=\SI{8}{\giga\hertz}$, it is very difficult to recognize the existence of $D_2$ and $D_3$ because several artifacts with large magnitudes are also included. Furthermore, it is very difficult to recognize the existence, location, and outline shape of objects due to the excessive number of artifacts at $f=\SI{12}{\giga\hertz}$.
\end{example}

\begin{figure}[h]
\begin{center}
\subfigure[$f=\SI{4}{\giga\hertz}$]{\includegraphics[width=.33\columnwidth]{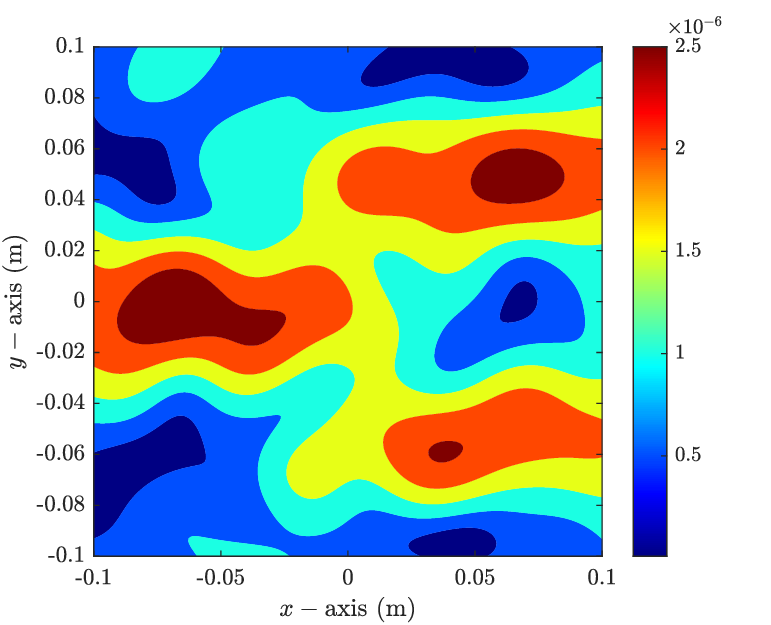}}\hfill
\subfigure[$f=\SI{8}{\giga\hertz}$]{\includegraphics[width=.33\columnwidth]{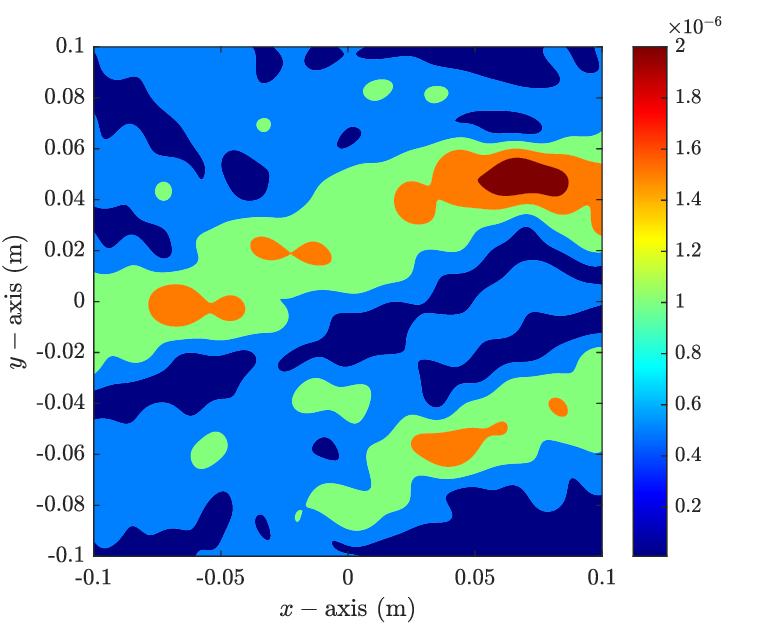}}\hfill
\subfigure[$f=\SI{12}{\giga\hertz}$]{\includegraphics[width=.33\columnwidth]{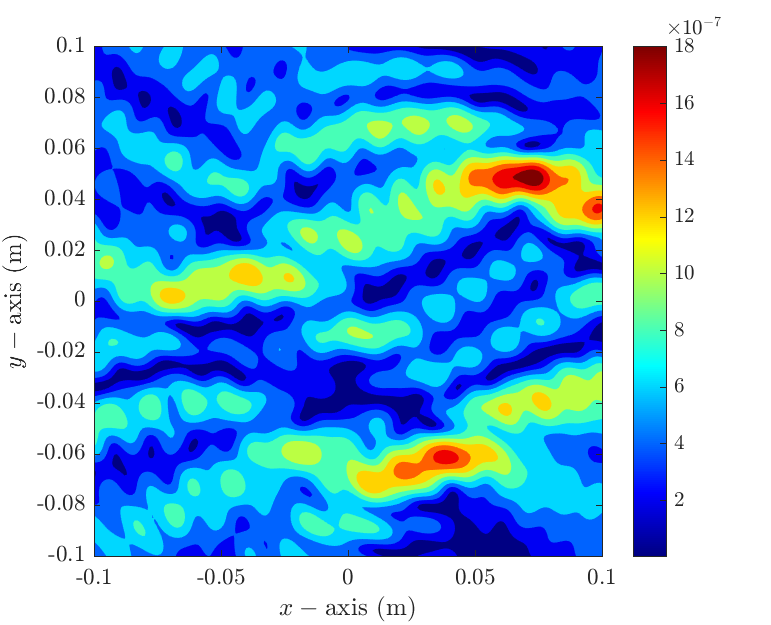}}
\caption{\label{Synthetic_Single_Case1}(Example \ref{Ex_Single_Case1}) Maps of $\mathfrak{F}_{\osm}(\mr,1)$ in Case $1$.}
\end{center}
\end{figure}

\begin{example}\label{Ex_Single_Case2}
Figure \ref{Synthetic_Single_Case2} shows maps of $\mathfrak{F}_{\osm}(\mr,10)$ with $\vv_{10}=\SI{0.72}{\meter}(\cos\SI{90}{\degree},\sin\SI{90}{\degree})$ in Case $2$. Since $\mu_1<\mu_2<\mu_3$ and $\alpha_s$ are same,
\[\frac{\mu_0}{\mu_1+\mu_0}>\frac{\mu_0}{\mu_2+\mu_0}>\frac{\mu_0}{\mu_3+\mu_0}\quad\text{implies}\quad\mathfrak{F}_{\osm}(\mr_1,10)>\mathfrak{F}_{\osm}(\mr_2,10)>\mathfrak{F}_{\osm}(\mr_3,10).\]
Hence, the value of $\mathfrak{F}_{\osm}(\mr,10)$ reached its maximum value at $\mr=\mr_1\in D_1$ so that the existence of $D_1$ can be recognized clearly. Fortunately, the existence of $D_2$ can be recognized also but since the value of $\mathfrak{F}_{\osm}(\mr,10)$ in the neighborhood of $D_3$ is small and difficult to distinguish from several artifacts in the map, it is challenging to recognize the existence of $D_3$. Furthermore, it can be observed that the shape of peaks with a large magnitude differs from that in Figure \ref{Synthetic_Single_Case1}. Roughly speaking, it can be seen that it spreads parallel to the direction vector $\vv_{10}$, which corresponds to the position of the source.
\end{example}
 
\begin{figure}[h]
\begin{center}
\subfigure[$f=\SI{4}{\giga\hertz}$]{\includegraphics[width=.33\columnwidth]{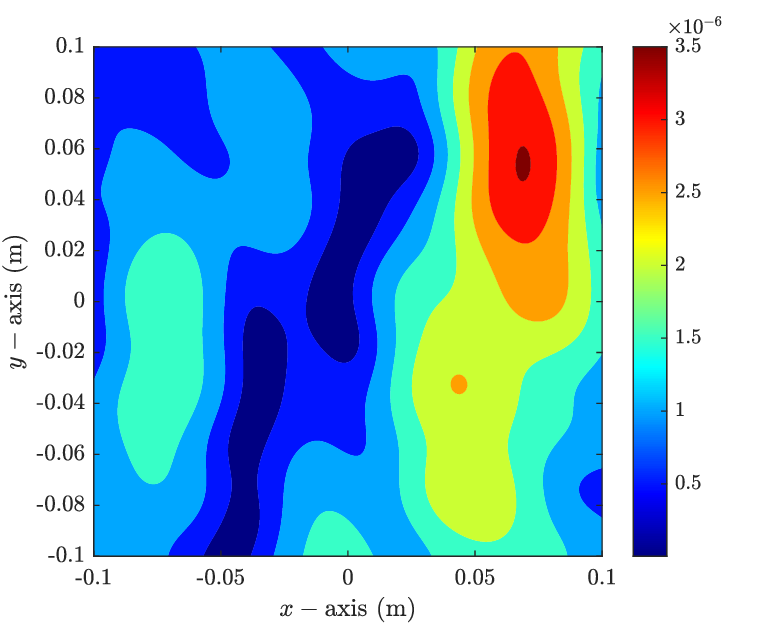}}\hfill
\subfigure[$f=\SI{8}{\giga\hertz}$]{\includegraphics[width=.33\columnwidth]{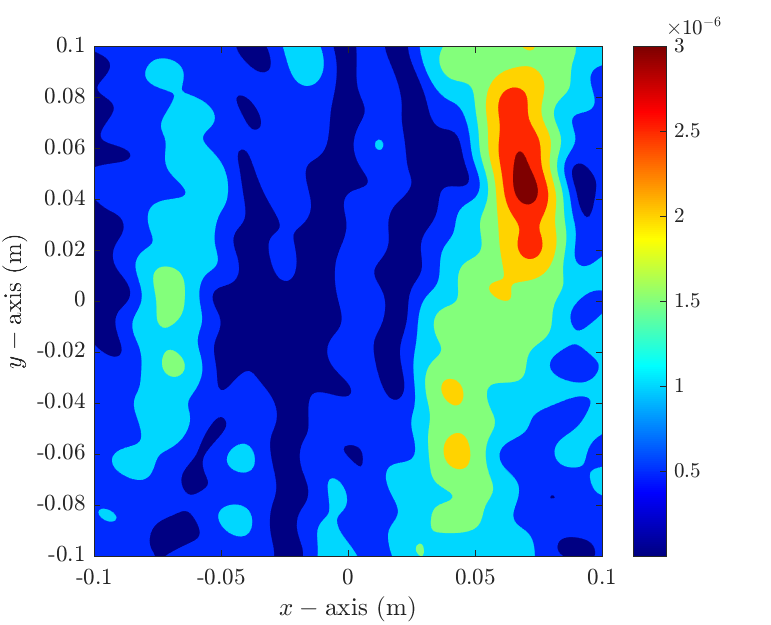}}\hfill
\subfigure[$f=\SI{12}{\giga\hertz}$]{\includegraphics[width=.33\columnwidth]{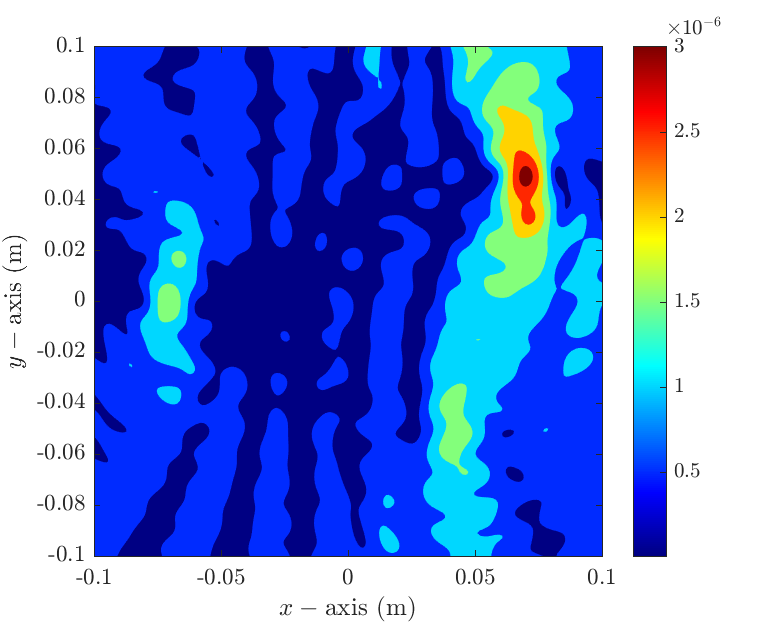}}
\caption{\label{Synthetic_Single_Case2}(Example \ref{Ex_Single_Case2}) Maps of $\mathfrak{F}_{\osm}(\mr,10)$ in Case $2$.}
\end{center}
\end{figure}

\begin{example}\label{Ex_Single_Case3}
Figure \ref{Synthetic_Single_Case3} shows maps of $\mathfrak{F}_{\osm}(\mr,25)$ with $\vv_{25}=\SI{0.72}{\meter}(\cos\SI{240}{\degree},\sin\SI{240}{\degree})$ in Case $3$. Since $\alpha_1>\alpha_2>\alpha_3$ and $\mu_s$ are same, $\mathfrak{F}_{\osm}(\mr_1,25)>\mathfrak{F}_{\osm}(\mr_2,25)>\mathfrak{F}_{\osm}(\mr_3,25)$ so that similar to the Example \ref{Ex_Single_Case2}, the existence of $D_1$ and $D_2$ can be recognized. However, since the value of $\mathfrak{F}_{\osm}(\mr,25)$ in the neighborhood of $D_3$ is very small, it is impossible to recognize the existence of $D_3$. Same as in Example \ref{Ex_Single_Case2}, the shape of peaks with a large magnitude at $\mr_1$ and $\mr_2$ spreads parallel to the direction vector $\vv_{25}$, which corresponds to the position of the source.
\end{example}

\begin{figure}[h]
\begin{center}
\subfigure[$f=\SI{4}{\giga\hertz}$]{\includegraphics[width=.33\columnwidth]{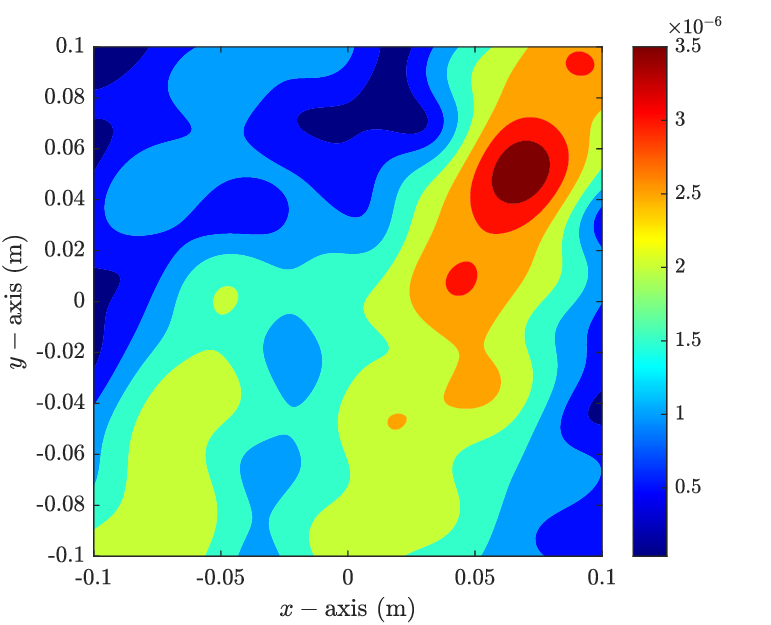}}\hfill
\subfigure[$f=\SI{8}{\giga\hertz}$]{\includegraphics[width=.33\columnwidth]{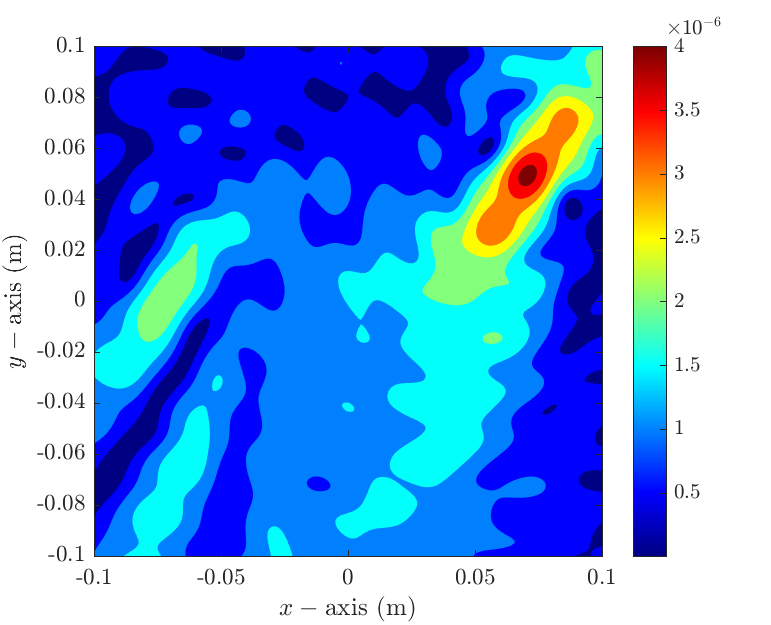}}\hfill
\subfigure[$f=\SI{12}{\giga\hertz}$]{\includegraphics[width=.33\columnwidth]{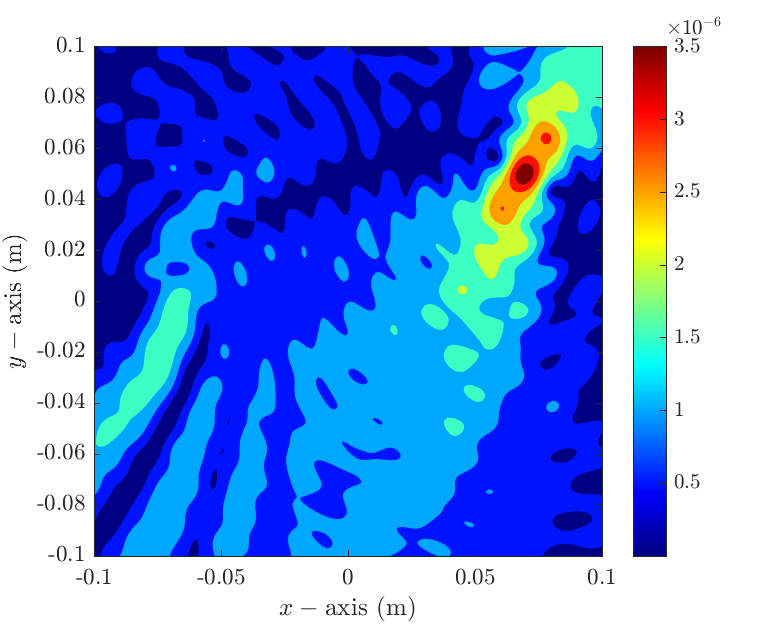}}
\caption{\label{Synthetic_Single_Case3}(Example \ref{Ex_Single_Case3}) Maps of $\mathfrak{F}_{\osm}(\mr,25)$ in Case $3$.}
\end{center}
\end{figure}

\begin{example}\label{Ex_Single_Case4}
Figure \ref{Synthetic_Single_Case4} shows maps of $\mathfrak{F}_{\osm}(\mr,1)$ in Case $4$. Notice that since
\[\alpha_1^2\left(\frac{\mu_0}{\mu_1+\mu_0}\right)=0.0056>\alpha_2^2\left(\frac{\mu_0}{\mu_2+\mu_0}\right)=0.0017\gg\alpha_3^2\left(\frac{\mu_0}{\mu_3+\mu_0}\right)=3.125\times10^{-4},\]
$\mathfrak{F}(\mr,1)$ in the neighborhood of $D_1$ is considerably larger than $\mathfrak{F}(\mr,1)$ in the neighborhoods of $D_2$ and $D_3$. Correspondingly, at $f=\SI{4}{\giga\hertz}$ the $D_1$ is recognizable but $D_2$ is barely recognizable. However, at $f=\SI{8}{\giga\hertz}$ and $f=\SI{12}{\giga\hertz}$, $D_2$ cannot be recognized. At any frequency of operation, $D_3$ cannot be recognized.
\end{example}

\begin{figure}[h]
\begin{center}
\subfigure[$f=\SI{4}{\giga\hertz}$]{\includegraphics[width=.33\columnwidth]{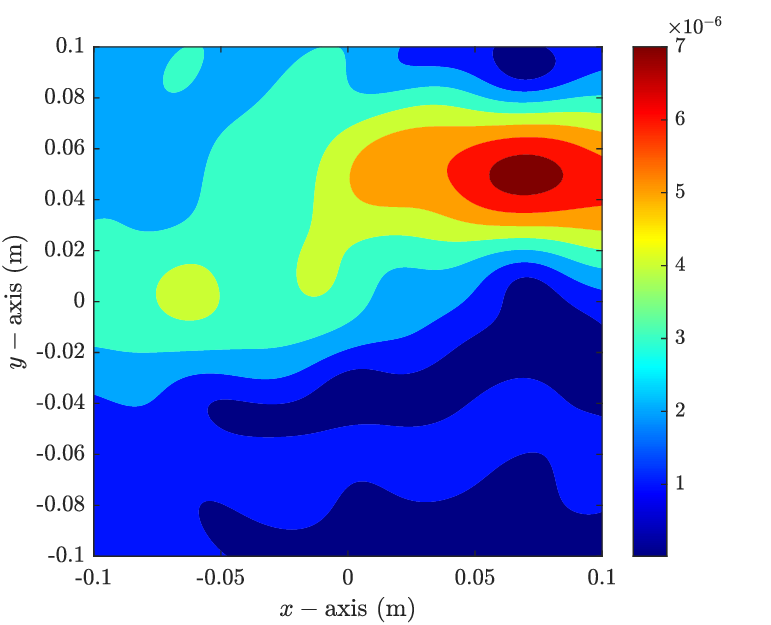}}\hfill
\subfigure[$f=\SI{8}{\giga\hertz}$]{\includegraphics[width=.33\columnwidth]{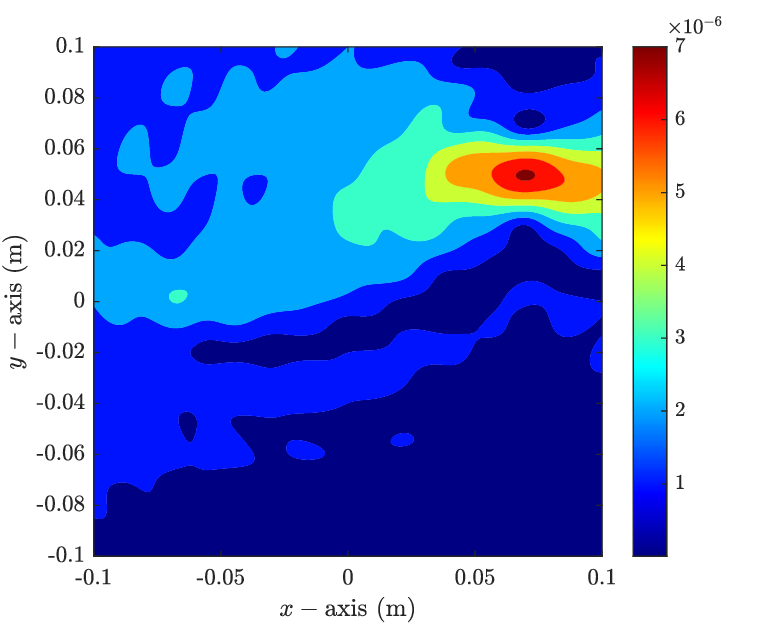}}\hfill
\subfigure[$f=\SI{12}{\giga\hertz}$]{\includegraphics[width=.33\columnwidth]{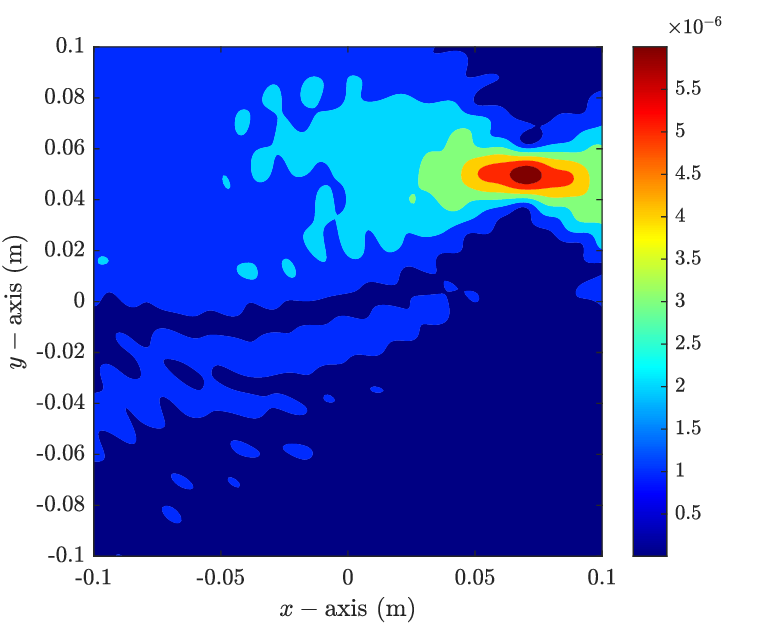}}
\caption{\label{Synthetic_Single_Case4}(Example \ref{Ex_Single_Case4}) Maps of $\mathfrak{F}_{\osm}(\mr,1)$ in Case $4$.}
\end{center}
\end{figure}

\subsection{Simulation results using Fresnel experimental dataset}
Here, we exhibit simulation results using 2D Fresnel experimental dataset \cite{BS}. The simulation setup is same as the synthetic dataset experiment while the object $D$ is a metallic target with a rectangular cross-section centered at the origin.

\begin{example}\label{Ex_Fresnel_Single1}
Figure \ref{Fresnel_Single1} shows the maps of $\mathfrak{F}_{\osm}(\mr,1)$ at frequencies $f=2,4,8,10,12,\SI{16}{\giga\hertz}$. Based on the imaging results, it is possible to recognize the existence of $D$; however, it is very difficult to recognize the shape of $D$ when $f=2,\SI{4}{\giga\hertz}$. Fortunately, when the applied frequency is sufficiently high ($f\geq\SI{8}{\giga\hertz}$), it is possible to recognize the outline shape of $D$.
\end{example}

\begin{figure}[h]
\begin{center}
\subfigure[$f=\SI{2}{\giga\hertz}$]{\includegraphics[width=.33\columnwidth]{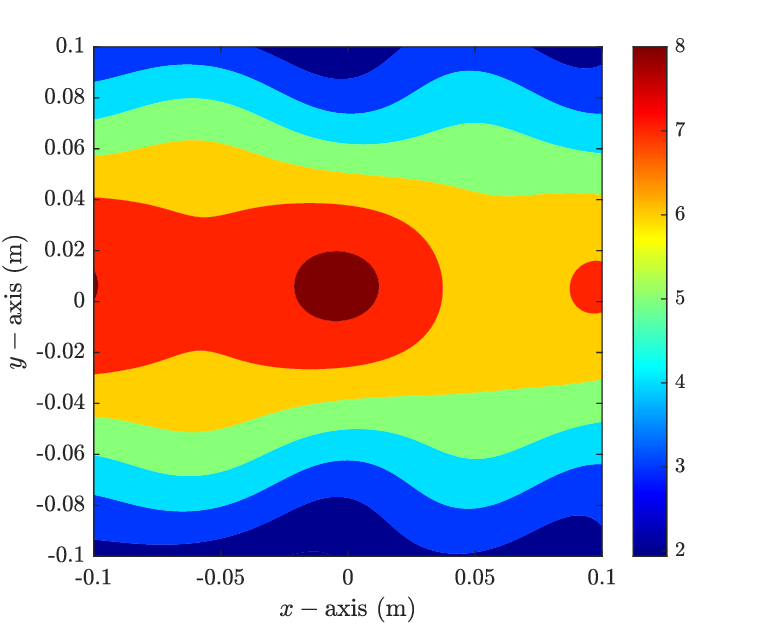}}\hfill
\subfigure[$f=\SI{4}{\giga\hertz}$]{\includegraphics[width=.33\columnwidth]{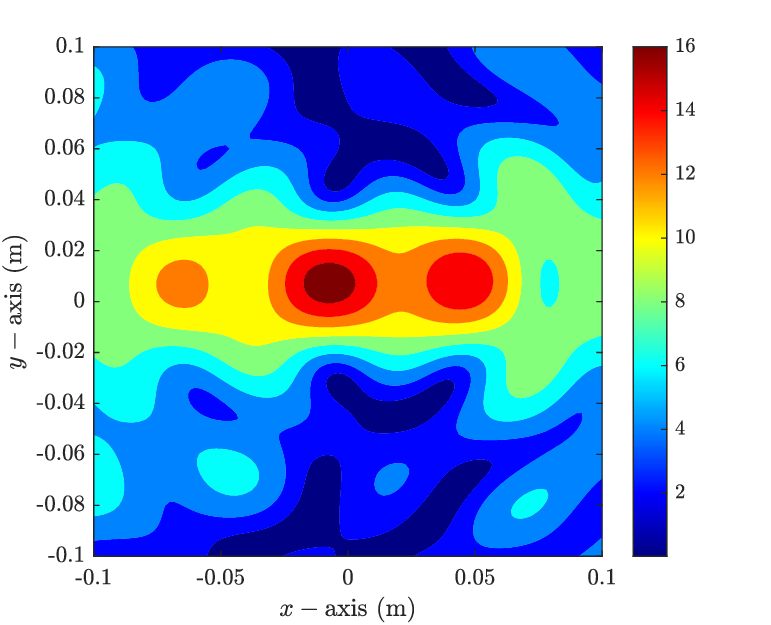}}\hfill
\subfigure[$f=\SI{8}{\giga\hertz}$]{\includegraphics[width=.33\columnwidth]{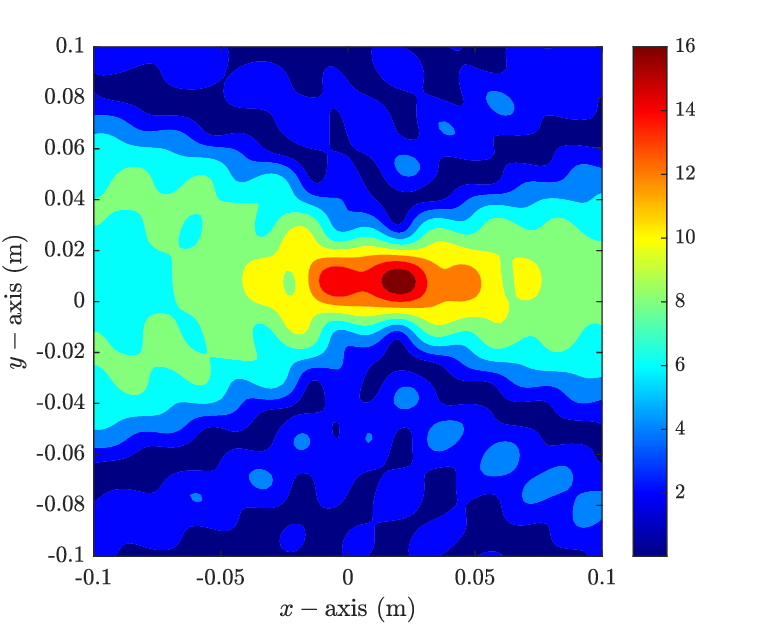}}\\
\subfigure[$f=\SI{10}{\giga\hertz}$]{\includegraphics[width=.33\columnwidth]{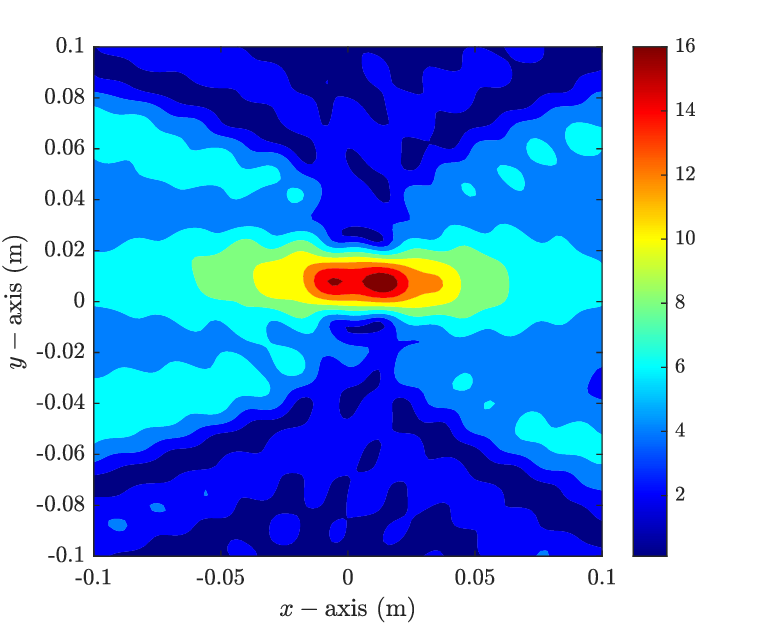}}\hfill
\subfigure[$f=\SI{12}{\giga\hertz}$]{\includegraphics[width=.33\columnwidth]{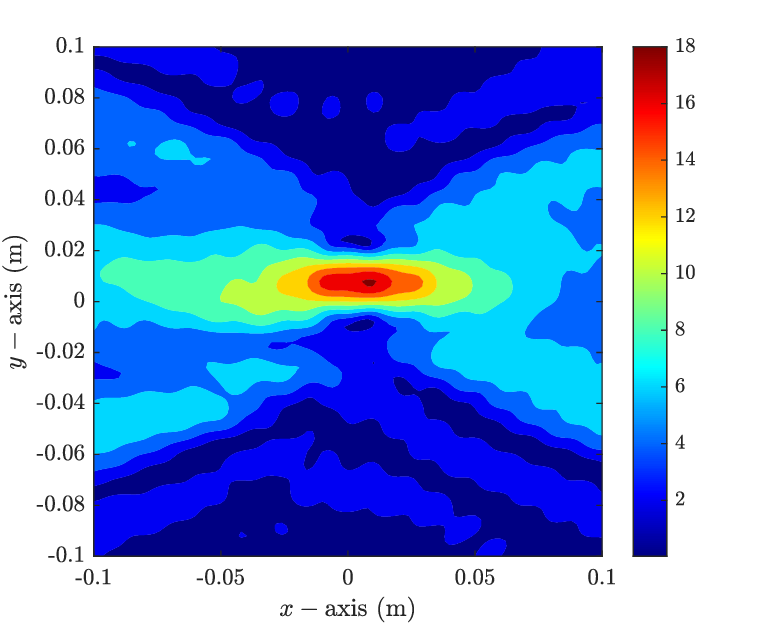}}\hfill
\subfigure[$f=\SI{16}{\giga\hertz}$]{\includegraphics[width=.33\columnwidth]{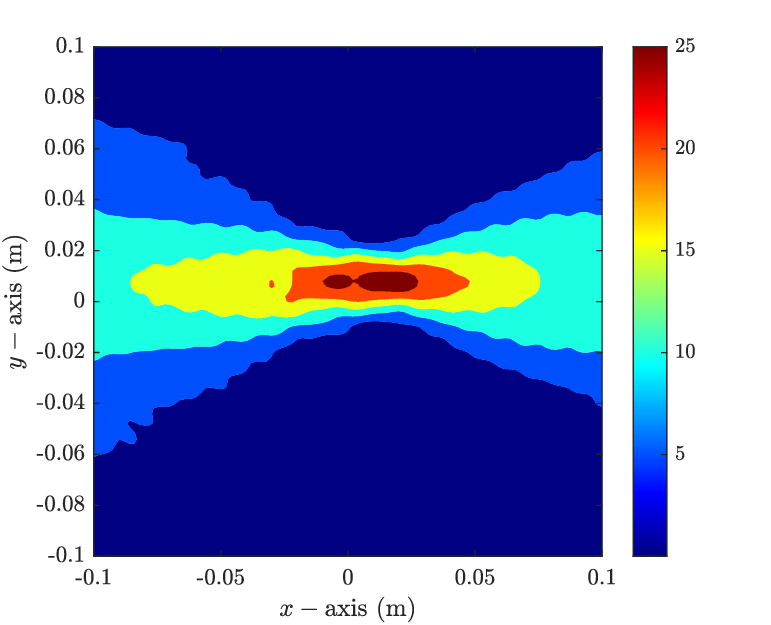}}
\caption{\label{Fresnel_Single1} (Example \ref{Ex_Fresnel_Single1}) Maps of $\mathfrak{F}_{\osm}(\mr,1)$.}
\end{center}
\end{figure}

\begin{example}\label{Ex_Fresnel_Single10}
Figure \ref{Fresnel_Single10} shows the maps of $\mathfrak{F}_{\osm}(\mr,10)$ at several frequencies. In contrast to the imaging with $\vv_1$ shown in Figure \ref{Fresnel_Single1}, here, it is impossible to recognize the existence of $D$ at $f=2,\SI{4}{\giga\hertz}$. Note that although the existence of $D$ can be recognized when the applied frequency is sufficiently high ($f\geq\SI{8}{\giga\hertz}$), the identified shape differs from the true one because the peak with a large magnitude spreads parallel to the direction vector $\vv_{10}$.
\end{example}

\begin{figure}[h]
\begin{center}
\subfigure[$f=\SI{2}{\giga\hertz}$]{\includegraphics[width=.33\columnwidth]{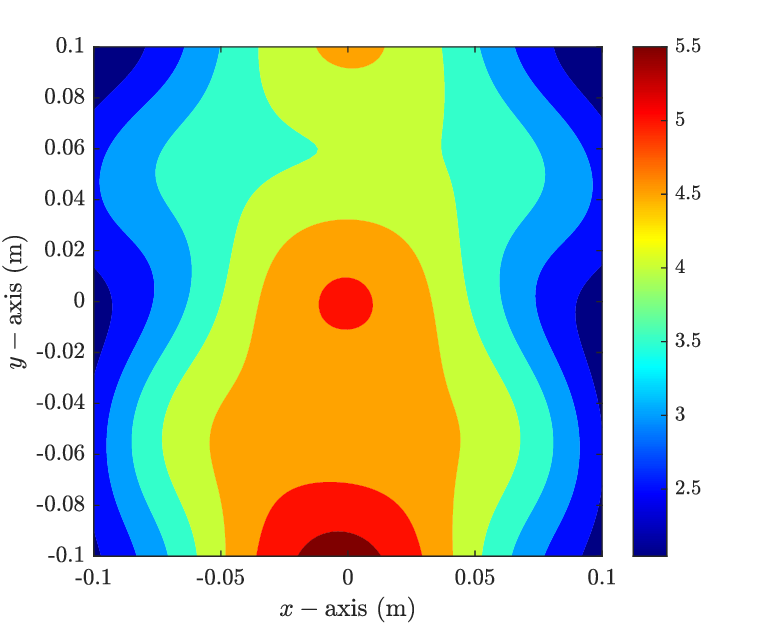}}\hfill
\subfigure[$f=\SI{4}{\giga\hertz}$]{\includegraphics[width=.33\columnwidth]{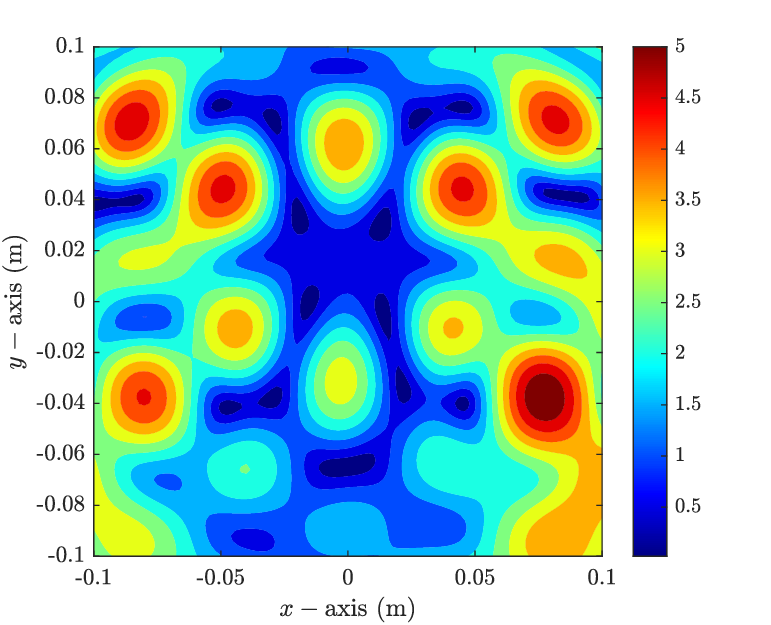}}\hfill
\subfigure[$f=\SI{8}{\giga\hertz}$]{\includegraphics[width=.33\columnwidth]{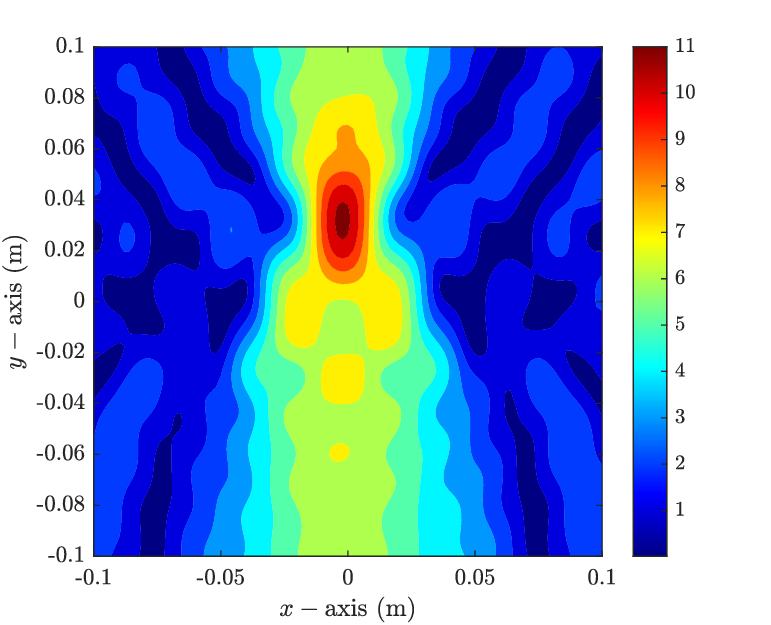}}\\
\subfigure[$f=\SI{10}{\giga\hertz}$]{\includegraphics[width=.33\columnwidth]{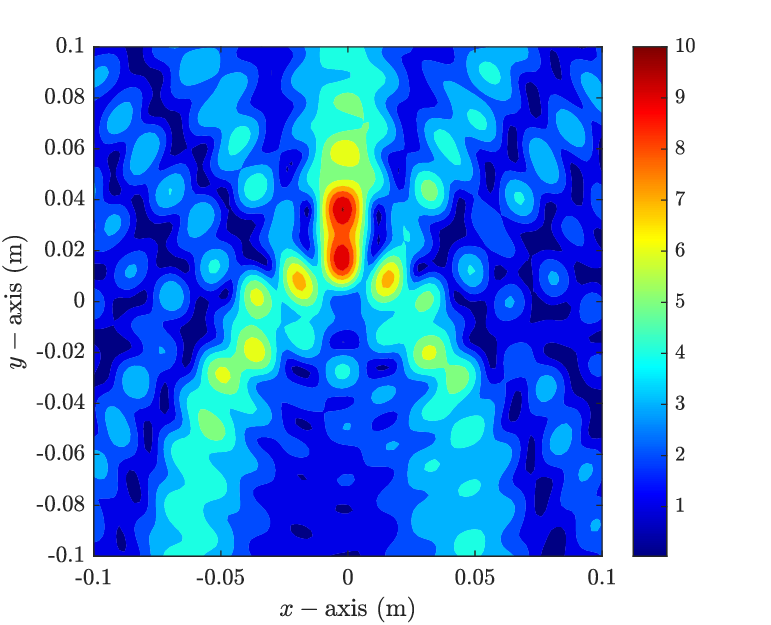}}\hfill
\subfigure[$f=\SI{12}{\giga\hertz}$]{\includegraphics[width=.33\columnwidth]{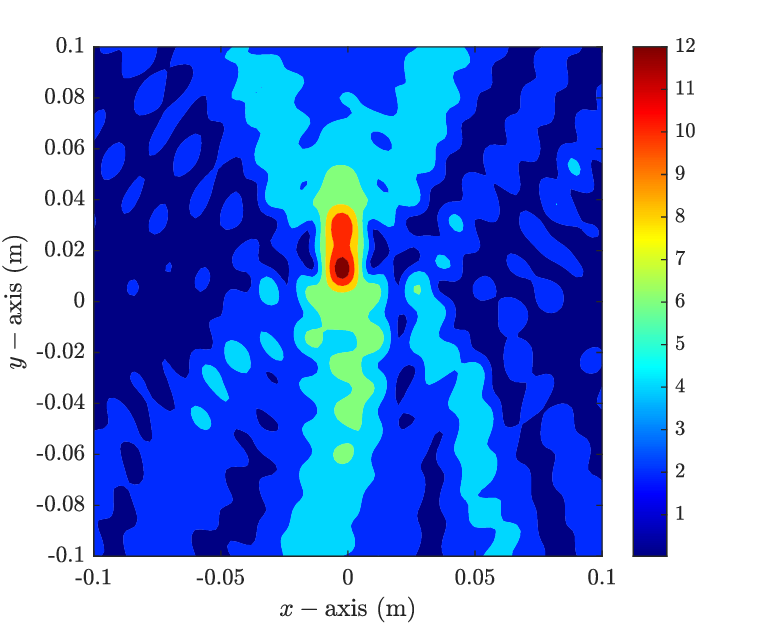}}\hfill
\subfigure[$f=\SI{16}{\giga\hertz}$]{\includegraphics[width=.33\columnwidth]{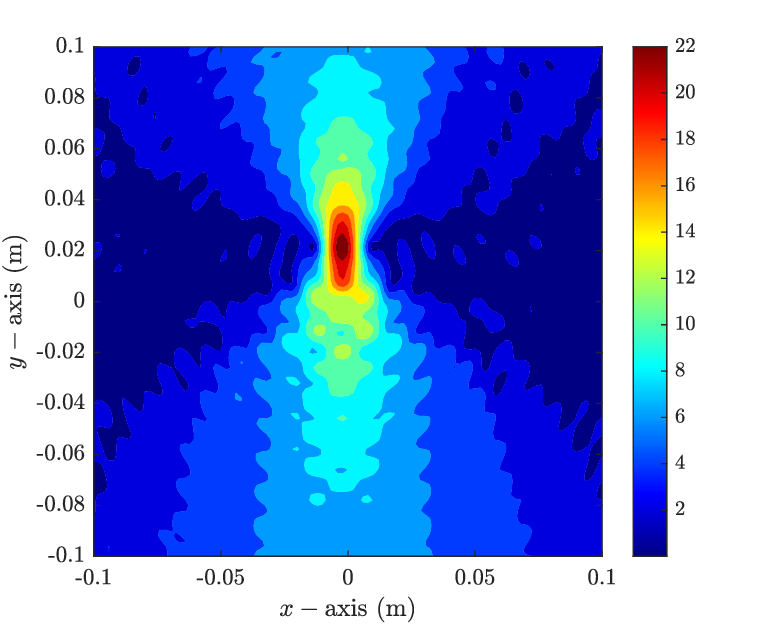}}
\caption{\label{Fresnel_Single10}(Example \ref{Ex_Fresnel_Single10}) Maps of $\mathfrak{F}_{\osm}(\mr,10)$.}
\end{center}
\end{figure}

\begin{example}\label{Ex_Fresnel_Single25}
Figure \ref{Fresnel_Single25} shows the maps of $\mathfrak{F}_{\osm}(\mr,25)$ at several frequencies. Based on the imaging results, similar phenomena in Examples \ref{Ex_Fresnel_Single1} and \ref{Ex_Fresnel_Single10} can be examined. While the outline shape of $D$ was detected in Example \ref{Ex_Fresnel_Single1} and the  approximate location of $D$ was retrieved in Example \ref{Ex_Fresnel_Single10}, no information about the $D$ can be obtained from the maps of $\mathfrak{F}_{\osm}(\mr,25)$. Roughly speaking, Only the presence of $D$ near the origin can be confirmed when $f\geq\SI{8}{\giga\hertz}$.
\end{example}

\begin{figure}[h]
\begin{center}
\subfigure[$f=\SI{2}{\giga\hertz}$]{\includegraphics[width=.33\columnwidth]{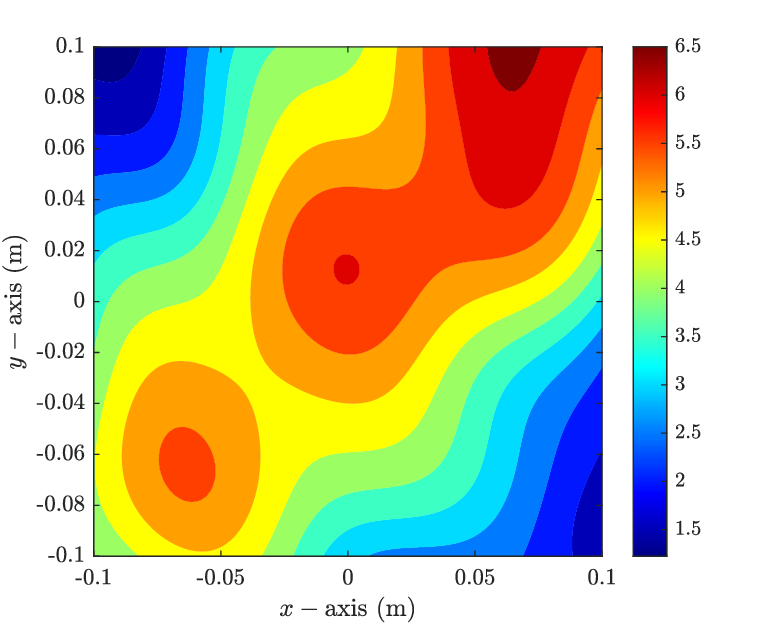}}\hfill
\subfigure[$f=\SI{4}{\giga\hertz}$]{\includegraphics[width=.33\columnwidth]{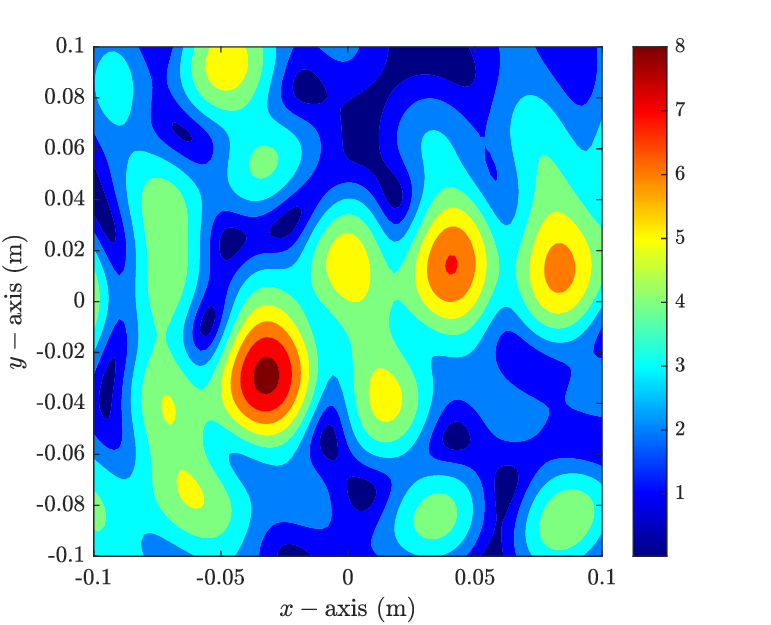}}\hfill
\subfigure[$f=\SI{8}{\giga\hertz}$]{\includegraphics[width=.33\columnwidth]{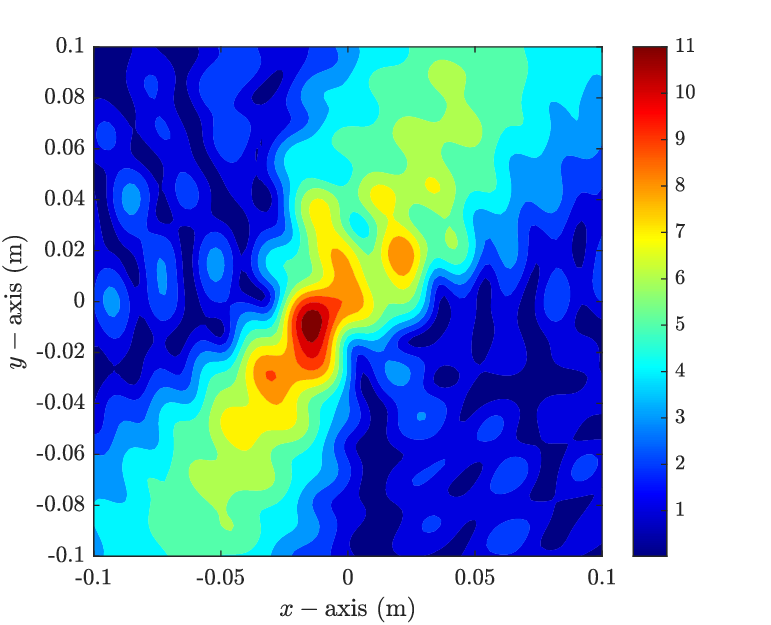}}\\
\subfigure[$f=\SI{10}{\giga\hertz}$]{\includegraphics[width=.33\columnwidth]{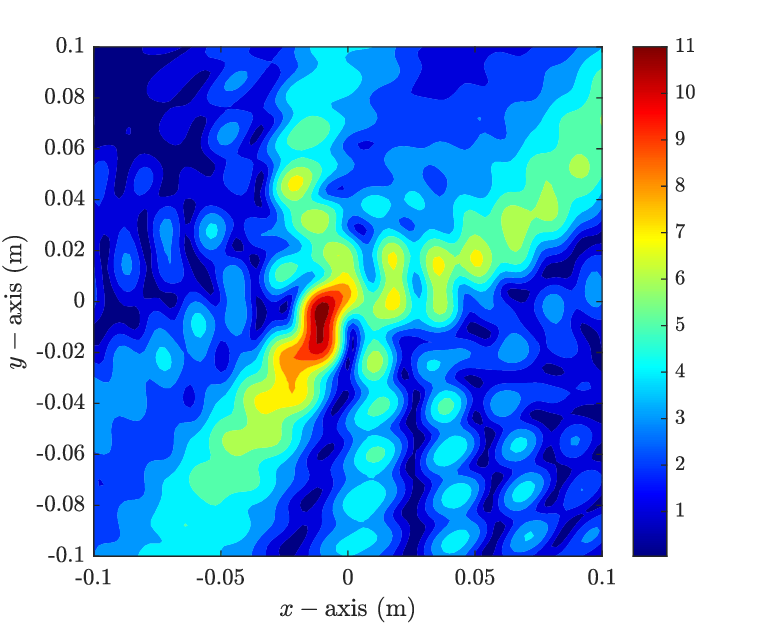}}\hfill
\subfigure[$f=\SI{12}{\giga\hertz}$]{\includegraphics[width=.33\columnwidth]{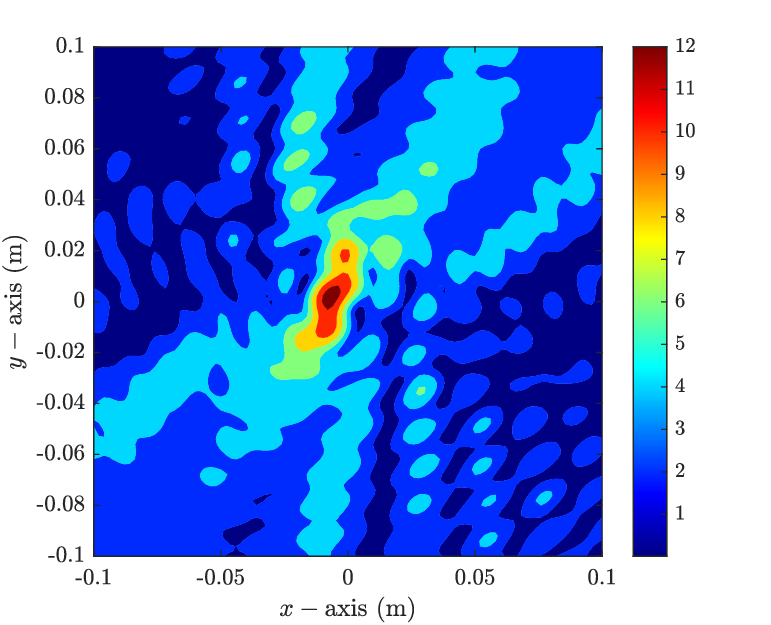}}\hfill
\subfigure[$f=\SI{16}{\giga\hertz}$]{\includegraphics[width=.33\columnwidth]{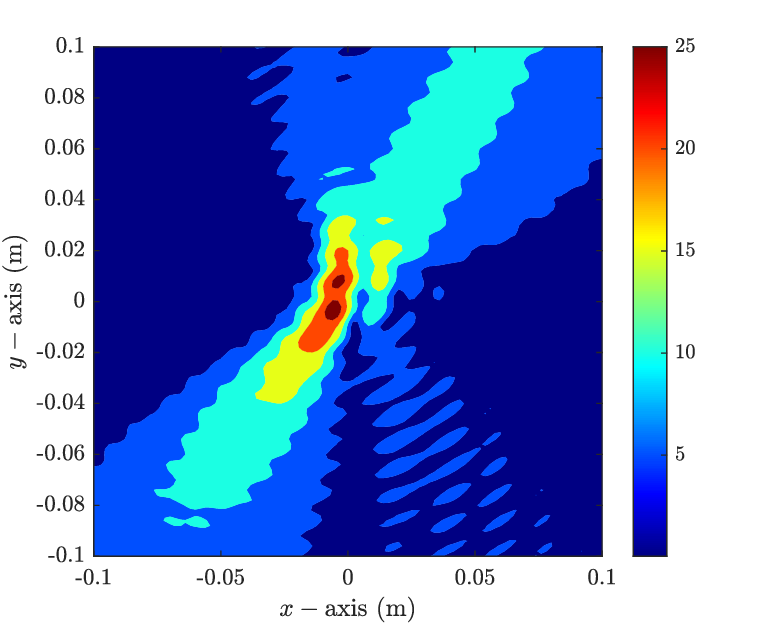}}
\caption{\label{Fresnel_Single25}(Example \ref{Ex_Fresnel_Single25}) Maps of $\mathfrak{F}_{\osm}(\mr,25)$.}
\end{center}
\end{figure}

Based on the simulation results, in addition to Remarks \ref{Remark_Single1} and \ref{Remark_Single4}, we can conclude that the imaging performance of the OSM with a single source is strongly dependent on the operated frequency and the location of the emitter. Thus, another indicator function must be developed with multiple sources to improve the imaging performance of the OSM.

\section{Imaging function with multiple sources}\label{sec:5}
Based on the theoretical and numerical simulation results in Sections \ref{sec:3} and \ref{sec:4}, further improvement is essential for a proper identification of small objects. Previous studies \cite{ACP,AGJKLSW,AGKPS,P-SUB3,P-SUB16,P-SUB18,P1} have confirmed that applying multiple sources and/or frequencies guarantees improved imaging performance. Here, we consider the single-frequency OSM with multiple sources, and we introduce the following indicator function with multiple sources:
\begin{equation}\label{ImagingFunction_Multiple}
\mathfrak{F}_{\msm}(\mr)=\abs{\sum_{m=1}^{M}\mE(m)\cdot\overline{\mG(\mr)}}=\left|\sum_{m=1}^{M}\sum_{n=1}^{N}u_{\scat}(\mb_{m,n},\ma_m)\overline{\nabla G(\mb_{m,N},\mr)\cdot\nabla G(\ma_m,\mr)}\right|.
\end{equation}
To demonstrate the feasibility and improvement of this multisource indicator function, we derive the following result. Prior to the derivation, we emphasize that, in contrast to the receiver setting, the range of the emitters is $\SI{0}{\degree}$ to $\SI{350}{\degree}$ with a step size of $\SI{10}{\degree}$.

\begin{theorem}\label{OSM_Multiple}
  Let $\vv_m=(\cos\vartheta_m,\sin\vartheta_m)$, $\vv=(\cos\vartheta,\sin\vartheta)$, and $\mr-\mr_s=|\mr-\mr_s|(\cos\phi_s,\sin\phi_s)$. Here, if $4k|\mr-\mb_{m,n}|\gg1$ and $4k|\mr-\ma_m|\gg1$ for all $m=1,2,\ldots,M$ and $n=1,2,\ldots,N$, then $\mathfrak{F}_{\msm}(\mr)$ can be expressed as follows:
  \begin{equation}\label{Structure_Multiple}
    \mathfrak{F}_{\msm}(\mr)=\left|\frac{MNk^2}{4AB}\sum_{s=1}^{S}\alpha_s^2\pi\left(\frac{\mu_0}{\mu_s+\mu_0}\right)\left[\left(\frac12-\frac{3\sqrt{3}}{16\pi}\right)J_0(k|\mr-\mr_s|)^2+J_2(k|\mr-\mr_s|)^2+\mathcal{E}_{\msm}(\mr)\right]\right|,
  \end{equation}
  where 
  \begin{align*}
  \mathcal{E}_{\msm}(\mr)&=\frac{3}{2\pi}\sum_{q=1}^{\infty}\frac{(-1)^{q}}{q}\sin\frac{2q\pi}{3}J_q(k|\mr-\mr_s|)^2\\
&+\frac{3}{2\pi}\sum_{q=1}^{\infty}\frac{(-1)^{q+1}}{q+2}\sin\frac{2(q+2)\pi}{3}J_q(k|\mr-\mr_s|)^2+\frac{3}{2\pi}\sum_{q=1,q\ne2}^{\infty}\frac{(-1)^{q-1}}{q-2}\sin\frac{2(q-2)\pi}{3}J_q(k|\mr-\mr_s|)^2.
  \end{align*}
\end{theorem}
\begin{proof}
Note that the following Jacobi–Anger expansion holds uniformly (see \cite{CK} for instance):
\[e^{ik\vv_m\cdot(\mr-\mr_s)}=J_0(k|\mr-\mr_s|)+2\sum_{p=1}^{\infty}i^p J_p(k|\mr-\mr_s|)\cos(p\vartheta_m-p\phi_s).\]
Thus, based on the structure \eqref{Structure_Single}, we obtain the following:
\begin{align*}
&\sum_{m=1}^{M}e^{ik\vv_m\cdot(\mr-\mr_s)}\left[\left(\frac12-\frac{3\sqrt{3}}{16\pi}\right)J_0(k|\mr-\mr_s|)-\cos(2\vartheta_m-2\phi_s)J_2(k|\mr-\mr_s|)+\mathcal{E}_{\osm}(\mr,m)\right]\\
&=\sum_{m=1}^{M}\left(J_0(k|\mr-\mr_s|)+2\sum_{p=1}^{\infty}i^p J_p(k|\mr-\mr_s|)\cos(p\vartheta_m-p\phi_s)\right)\\
&\phantom{=}\times\left[\left(\frac12-\frac{3\sqrt{3}}{16\pi}\right)J_0(k|\mr-\mr_s|)-\cos(2\vartheta_m-2\phi_s)J_2(k|\mr-\mr_s|)+\mathcal{E}_{\osm}(\mr,m)\right]\\
&=\mathcal{I}_1+\mathcal{I}_2+\mathcal{I}_3+\mathcal{I}_4+\mathcal{I}_5+\mathcal{I}_6,
\end{align*}
where
\begin{align*}
\mathcal{I}_1&=\sum_{m=1}^{M}\left(\frac12-\frac{3\sqrt{3}}{16\pi}\right)J_0(k|\mr-\mr_s|)^2,\\
\mathcal{I}_2&=-\sum_{m=1}^{M}\cos(2\vartheta_m-2\phi_s)J_0(k|\mr-\mr_s|)J_2(k|\mr-\mr_s|),\\
\mathcal{I}_3&=J_0(k|\mr-\mr_s|)\sum_{m=1}^{M}\mathcal{E}_{\osm}(\mr,m),\\
\mathcal{I}_4&=\left(1-\frac{3\sqrt{3}}{8\pi}\right)\sum_{p=1}^{\infty}i^pJ_0(k|\mr-\mr_s|)J_p(k|\mr-\mr_s|)\sum_{m=1}^{M}\cos(p\vartheta_m-p\phi_s),\\
\mathcal{I}_5&=-2\sum_{p=1}^{\infty}i^pJ_2(k|\mr-\mr_s|)J_p(k|\mr-\mr_s|)\sum_{m=1}^{M}\cos(2\vartheta_m-2\phi_s)\cos(p\vartheta_m-p\phi_s),\\
\mathcal{I}_6&=2\sum_{p=1}^{\infty}i^p J_p(k|\mr-\mr_s|)\sum_{m=1}^{M}\mathcal{E}_{\osm}(\mr,m)\cos(p\vartheta_m-p\phi_s).
\end{align*}

First, direct calculation yields:
\begin{align}
\begin{aligned}\label{I1234}
\mathcal{I}_1=&M\left(\frac12-\frac{3\sqrt{3}}{16\pi}\right)J_0(k|\mr-\mr_s|)^2,\\
\mathcal{I}_2\approx&-\frac{M}{2\pi}J_0(k|\mr-\mr_s|)J_2(k|\mr-\mr_s|)\int_0^{2\pi}\cos(2\vartheta-2\phi_s)d\vartheta=0,\\
\mathcal{I}_3\approx&\frac{3M}{4\pi^2}\sum_{q=1}^{\infty}\frac{i^q}{q}\sin\frac{2q\pi}{3}J_0(k|\mr-\mr_s|)J_q(k|\mr-\mr_s|)\int_0^{2\pi}\cos(q\vartheta-q\phi_s)d\vartheta\\
&+\frac{3M}{4\pi^2}\sum_{q=1}^{\infty}\frac{(-i)^q}{q+2}\sin\frac{2(q+2)\pi}{3}J_0(k|\mr-\mr_s|)J_q(k|\mr-\mr_s|)\int_0^{2\pi}\cos(q\vartheta-q\phi_s)d\vartheta\\
&+\frac{3M}{4\pi^2}\sum_{q=1,q\ne2}^{\infty}\frac{(-i)^q}{q-2}\sin\frac{2(q-2)\pi}{3}J_0(k|\mr-\mr_s|)J_q(k|\mr-\mr_s|)\int_0^{2\pi}\cos(q\vartheta-q\phi_s)d\vartheta=0,\\
\mathcal{I}_4\approx&\frac{M}{2\pi}\left(1-\frac{3\sqrt{3}}{8\pi}\right)\sum_{p=1}^{\infty}i^pJ_0(k|\mr-\mr_s|)J_p(k|\mr-\mr_s|)\int_0^{2\pi}\cos(p\vartheta-p\phi_s)d\vartheta=0.
\end{aligned}
\end{align}

Next, for any integers $p$ and $q$, since
\[\int_0^{2\pi}\cos(p\vartheta+\alpha)\cos(q\vartheta-q\phi_s)d\vartheta=\left\{\begin{array}{cl}
\pi,&p=q\\
0,&\text{otherwise},
\end{array}\right.\]
we have
\begin{align}
\begin{aligned}\label{I5}
\mathcal{I}_5&\approx-\frac{M}{\pi}\sum_{p=1}^{\infty}i^pJ_2(k|\mr-\mr_s|)J_p(k|\mr-\mr_s|)\int_0^{2\pi}\cos(2\vartheta-2\phi_s)\cos(p\vartheta-p\phi_s)d\vartheta=MJ_2(k|\mr-\mr_s|)^2
\end{aligned}
\end{align}
and
\begin{align}
\begin{aligned}\label{I6}
\mathcal{I}_6\approx&\frac{3M}{2\pi^2}\sum_{p=1}^{\infty}\sum_{q=1}^{\infty}\frac{i^{p+q}}{q}\sin\frac{2q\pi}{3}J_p(k|\mr-\mr_s|)J_q(k|\mr-\mr_s|)\int_0^{2\pi}\cos(p\vartheta-p\phi_s)\cos(q\vartheta-q\phi_s)d\vartheta\\
&+\frac{3M}{2\pi^2}\sum_{p=1}^{\infty}\sum_{q=1}^{\infty}\frac{i^{p+q+2}}{q+2}\sin\frac{2(q+2)\pi}{3}J_p(k|\mr-\mr_s|)J_q(k|\mr-\mr_s|)\int_0^{2\pi}\cos(p\vartheta-p\phi_s)\cos(q\vartheta-q\phi_s)d\vartheta\\
&+\frac{3M}{2\pi^2}\sum_{p=1}^{\infty}\sum_{q=1,q\ne2}^{\infty}\frac{i^{p+q-2}}{q-2}\sin\frac{2(q-2)\pi}{3}J_p(k|\mr-\mr_s|)J_q(k|\mr-\mr_s|)\int_0^{2\pi}\cos(p\vartheta-p\phi_s)\cos(q\vartheta-q\phi_s)d\vartheta\\
=&\frac{3M}{2\pi}\sum_{q=1}^{\infty}\frac{(-1)^{q}}{q}\sin\frac{2q\pi}{3}J_q(k|\mr-\mr_s|)^2+\frac{3M}{2\pi}\sum_{p=1}^{\infty}\frac{(-1)^{q+1}}{q+2}\sin\frac{2(q+2)\pi}{3}J_q(k|\mr-\mr_s|)^2\\
&+\frac{3M}{2\pi}\sum_{q=1,q\ne2}^{\infty}\frac{(-1)^{q-1}}{q-2}\sin\frac{2(q-2)\pi}{3}J_q(k|\mr-\mr_s|)^2.
\end{aligned}
\end{align}
Thus, by combining \eqref{I1234}, \eqref{I5}, and \eqref{I6}, we obtain
\[\sum_{m=1}^{M}\mE(m)\cdot\overline{\mG(\mr)}=\frac{MNk^2}{4AB}\sum_{s=1}^{S}\alpha_s^2\pi\left(\frac{\mu_0}{\mu_s+\mu_0}\right)\left[\left(\frac12-\frac{3\sqrt{3}}{16\pi}\right)J_0(k|\mr-\mr_s|)^2+J_2(k|\mr-\mr_s|)^2+\mathcal{E}_{\msm}(\mr)\right]\]
and correspondingly, \eqref{Structure_Multiple} can be obtained.
\end{proof}

Based on Theorems \ref{OSM_Single} and \ref{OSM_Multiple}, we can obtain the Corollary \ref{Corollary_Comparison}. Before the derivation, we examine the following relationships.

\begin{lemma}\label{LemmaInequalities}For sufficiently large $Q$, the following inequalities hold
\begin{equation}\label{SeriesInequality}
\mathcal{S}_1(Q)<\mathcal{S}_2(Q).
\end{equation}
where
\[\mathcal{S}_1(Q)=\sum_{q=1,q\ne 2}^{Q}\frac{1}{q-2}+\sum_{q=1}^{Q}\frac{1}{q+2}\quad\text{and}\quad\mathcal{S}_2(Q)=\sum_{q=1}^{Q}\frac{2}{q}.\]
\end{lemma}
\begin{proof}
Based on the Euler–Maclaurin formula,
\[\sum_{q=1}^{Q}\frac{1}{q}=\ln Q+\gamma+\frac{1}{2Q}+c_Q=\ln Q+\gamma+\frac{1}{2Q}+O\left(\frac{1}{Q^2}\right),\quad0\leq c_Q \leq\frac{1}{8Q^2},\]
where $\gamma=0.577215665\ldots$ denotes the Euler-Mascheroni constant. Since
\begin{align*}
&\sum_{q=1,q\ne 2}^{Q}\frac{1}{q-2}=\ln(Q-2)+\gamma+\frac{1}{2(Q-2)}+O\left(\frac{1}{Q^2}\right)-1\\
&\sum_{q=1}^{Q}\frac{1}{q+2}=\ln(Q+2)+\gamma+\frac{1}{2(Q+2)}+O\left(\frac{1}{Q^2}\right)-\frac{7}{12},
\end{align*}
we can examine that
\[\mathcal{S}_2(Q)-\mathcal{S}_1(Q)=\ln\frac{Q^2}{Q^2-4}+\frac{1}{Q}-\frac{1}{2(Q^2-4)}+\frac{19}{12}+O\left(\frac{1}{Q^2}\right)>0.\]
Hence, the relationship \eqref{SeriesInequality} holds.
\end{proof}

\begin{remark}
In order to examine \eqref{SeriesInequality}, we refer to Table \ref{TableSeries} for numerical values of $\mathcal{S}_1(Q)$ and $\mathcal{S}_2(Q)$ with various $Q$.
\end{remark}

\begin{table}[h]           
\begin{center}
\begin{tabular}{crrrrrrrrr}\hline
$Q$&$10^1$&$10^2$&$10^3$&$10^4$&$10^5$&$10^6$&$10^7$&$10^8$&$10^9$\\\hline
$\mathcal{S}_1(Q)$&$3.3211$&$7.8744$&$12.4709$&$17.0752$&$21.6803$&$26.2855$&$30.8906$&$35.4958$&$40.1010$\\
$\mathcal{S}_2(Q)$&$5.8579$&$10.3748$&$14.9709$&$19.5752$&$24.1803$&$28.7855$&$33.3906$&$37.9958$&$42.6010$\\\hline
\end{tabular}
\caption{\label{TableSeries}Values of $\mathcal{S}_1(Q)$ and $\mathcal{S}_2(Q)$ with various $Q$.}
\end{center}
\end{table}

\begin{corollary}\label{Corollary_Comparison}
Under the same conditions as Theorems \ref{OSM_Single} and \ref{OSM_Multiple}, the following statements hold: for $m=1,2,\ldots,M$,
\begin{enumerate}
\item if $\mr$ is far away from $D_s$ then $|\mathcal{E}_{\osm}(\mr,m)|$ and $|\mathcal{E}_{\msm}(\mr)|$ are negligible,
\item if $\mr$ is close to $\mr_s$ then $|\mathcal{E}_{\msm}(\mr)|\ll|\mathcal{E}_{\osm}(\mr,m)|$.
\end{enumerate}
\end{corollary}
\begin{proof}
Since $\mathcal{E}_{\osm}(\mr,m)$ and $\mathcal{E}_{\msm}(\mr)$ are uniformly convergent, for each $\epsilon>0$, there exists a natural number $Q=Q(\epsilon)$ such that
\begin{multline*}
\left|\mathcal{E}_{\osm}(\mr,m)-\frac{3}{2\pi}\sum_{q=1}^{Q}\left(\frac{i^q}{q}\sin\frac{2q\pi}{3}\cos(\vartheta_m-\phi_s)+\frac{(-i)^q}{q+2}\sin\frac{2(q+2)\pi}{3}\cos(q\vartheta_m-q\phi_s)\right)J_q(k|\mr-\mr_s|)\right.\\
\left.-\frac{3}{2\pi}\sum_{q=1,q\ne2}^{Q}\frac{(-i)^q}{q-2}\sin\frac{2(q-2)\pi}{3}\cos(q\vartheta_m-q\phi_s)J_q(k|\mr-\mr_s|)\right|<\epsilon
\end{multline*}
and
\begin{multline*}
\left|\mathcal{E}_{\msm}(\mr)-\frac{3}{2\pi}\sum_{q=1}^{Q}\frac{(-1)^q}{q}\sin\frac{2q\pi}{3}J_q(k|\mr-\mr_s|)^2-\frac{3}{2\pi}\sum_{q=1}^{Q}\frac{(-1)^{q+1}}{q+2}\sin\frac{2(q+2)\pi}{3}J_q(k|\mr-\mr_s|)^2\right.\\
\left.-\frac{3}{2\pi}\sum_{q=1,q\ne2}^{\infty}\frac{(-1)^{q-1}}{q-2}\sin\frac{2(q-2)\pi}{3}J_q(k|\mr-\mr_s|)^2\right|<\epsilon.
  \end{multline*}

Suppose that $\mr'$ is not close to $\mr$ such that $k|\mr-\mr'|\gg0.25$. Then, since the following asymptotic form holds
\[J_q(k|\mr-\mr'|)\approx\sqrt{\frac{2}{k|\mr-\mr'|}}\cos\left(k|\mr-\mr'|-\frac{q\pi}{2}-\frac{\pi}{4}+O\left(\frac{1}{k|\mr-\mr'|}\right)\right)\leq\sqrt{\frac{2}{k|\mr-\mr'|}},\]
applying the Euler-Maclaurin formula and \eqref{SeriesInequality} yields
\begin{align*}
|\mathcal{E}_{\osm}(\mr,m)|&\leq\frac{3}{2\pi}\left|\sum_{q=1}^{Q}\left(\frac{1}{q}\sin\frac{2q\pi}{3}+\sum_{q=1}^{Q}\frac{1}{q+2}\sin\frac{2(q+2)\pi}{3}\right)J_q(k|\mr-\mr_s|)+\sum_{q=1,q\ne2}^{Q}\frac{J_q(k|\mr-\mr_s|)}{q-2}\sin\frac{2(q-2)\pi}{3}\right|\\
&\leq\frac{3\sqrt{3}}{4\pi}\sqrt{\frac{2}{k|\mr-\mr'|}}\left|\sum_{q=1}^{Q}\frac{1}{q}+\sum_{q=1}^{Q}\frac{1}{q+2}+\sum_{q=1,q\ne2}^{Q}\frac{1}{q-2}\right|\leq\frac{9\sqrt{3}}{4\pi}\sqrt{\frac{2}{k|\mr-\mr'|}}\sum_{q=1}^{Q}\frac{1}{q}\\
&=\frac{9\sqrt{3}}{4\pi}\sqrt{\frac{2}{k|\mr-\mr'|}}\left(\ln Q+\gamma+\frac{1}{2Q}-c_Q\right)\leq\frac{9\sqrt{3}}{4\pi}\sqrt{\frac{2}{k|\mr-\mr'|}}\left(\ln Q+\gamma+\frac{1}{2Q}\right)
\end{align*}
and
\begin{align*}
|\mathcal{E}_{\msm}(\mr)|&\leq\frac{3}{2\pi}\left|\sum_{q=1}^{Q}\left(\frac{1}{q}\sin\frac{2q\pi}{3}+\sum_{q=1}^{Q}\frac{1}{q+2}\sin\frac{2(q+2)\pi}{3}\right)J_q(k|\mr-\mr_s|)^2+\sum_{q=1,q\ne2}^{Q}\frac{J_q(k|\mr-\mr_s|)^2}{q-2}\sin\frac{2(q-2)\pi}{3}\right|\\
&\leq\frac{3\sqrt{3}}{2\pi k|\mr-\mr'|}\left|\sum_{q=1}^{Q}\frac{1}{q}+\sum_{q=1}^{Q}\frac{1}{q+2}+\sum_{q=1,q\ne2}^{Q}\frac{1}{q-2}\right|\leq\frac{9\sqrt{3}}{2\pi k|\mr-\mr'|}\left(\ln Q+\gamma+\frac{1}{2Q}\right).
\end{align*}
Since
\[\frac{1}{k|\mr-\mr'|}\ll O(1),\]
we can say that the values of $\mathcal{E}_{\osm}(\mr,m)$ and $\mathcal{E}_{\msm}(\mr)$ are negligible for all $m=1,2,\ldots,M$, if the search point is far away from the objects.

Now, assume that $\mr$ is close to $\mr'$ such that $0<k|\mr-\mr'|<\delta\ll o(1)$. In this case, based on the asymptotic form of the Bessel function
\[J_q(k|\mr-\mr'|)\approx\frac{1}{\Gamma(q+1)}\left(\frac{k|\mr-\mr'|}{2}\right)^q,\]
we can examine the following relationships by applying \eqref{SeriesInequality}
\begin{align*}
|\mathcal{E}_{\osm}(\mr,m)|&\leq\frac{3\sqrt{3}}{4\pi}\left|\sum_{q=1}^{Q}\left(\frac{1}{q}+\sum_{q=1}^{Q}\frac{1}{q+2}\right)\frac{1}{\Gamma(q+1)}\left(\frac{k|\mr-\mr'|}{2}\right)^q+\sum_{q=1,q\ne2}^{Q}\frac{1}{(q-2)\Gamma(q+1)}\left(\frac{k|\mr-\mr'|}{2}\right)^q\right|\\
&<\frac{9\sqrt{3}}{4\pi}\sum_{q=1}^{Q}\frac{1}{q\Gamma(q+1)}\left(\frac{k|\mr-\mr'|}{2}\right)^q:=\sum_{q=1}^{Q}a_q
\end{align*}
and
\begin{align*}
|\mathcal{E}_{\msm}(\mr)|&\leq\frac{3\sqrt{3}}{4\pi}\left|\sum_{q=1}^{Q}\left(\frac{1}{q}+\sum_{q=1}^{Q}\frac{1}{q+2}\right)\frac{1}{\Gamma(q+1)^2}\left(\frac{k|\mr-\mr'|}{2}\right)^{2q}+\sum_{q=1,q\ne2}^{Q}\frac{1}{(q-2)\Gamma(q+1)^2}\left(\frac{k|\mr-\mr'|}{2}\right)^{2q}\right|\\
&<\frac{9\sqrt{3}}{4\pi}\sum_{q=1}^{Q}\frac{1}{q\Gamma(q+1)^2}\left(\frac{k|\mr-\mr'|}{2}\right)^{2q}:=\sum_{q=1}^{Q}b_q,
\end{align*}
where $\Gamma(q)=(q-1)!$ denotes the Gamma function. Since
\[\lim_{q\to\infty}\frac{b_q}{a_q}=\lim_{q\to\infty}\frac{(k|\mr-\mr'|)^q}{2^q\Gamma(q+1)}\ll\lim_{q\to\infty}\frac{\delta^q}{2^q\Gamma(q+1)}=0.\]
Hence, we can examine that $|\mathcal{E}_{\msm}(\mr)|\ll|\mathcal{E}_{\osm}(\mr,m)|$ for all $m=1,2,\ldots,M$.
\end{proof}

Based on Theorem \ref{OSM_Multiple} and Corollary \ref{Corollary_Comparison}, we can examine some properties of the indicator function $\mathfrak{F}_{\msm}(\mr)$.

\begin{remark}[Availability of object detection]\label{RemarkM1}
In contrast to the $\mathfrak{F}_{\osm}(\mr,m)$, the value of $\mathfrak{F}_{\msm}(\mr)$ is independent of $\mathcal{A}_m$. Here, we consider the following quantity:
\[\mathcal{D}_{\msm}^{(1)}(x)=\left(\frac12-\frac{3\sqrt{3}}{16\pi}\right)J_0(k|x|)^2+J_2(k|x|)^2.\]
Then based on numerical computation, we can observe that the value of $\mathcal{D}_{\msm}^{(1)}(x)$ reaches its global maximum value when $x=0$ (see Figure \ref{PlotMSM1}) so that, the value of $\mathfrak{F}_{\msm}(\mr)$ will reach its global maximum value when $\mr=\mr_s\in D_s$. As a result, it will be possible to recognize the existence and location of the objects. In addition, two peaks of large magnitudes will appear in the neighborhood of the objects due to the influence of the factor $J_2(k|\mr-\mr_s|)^2$.
\end{remark}

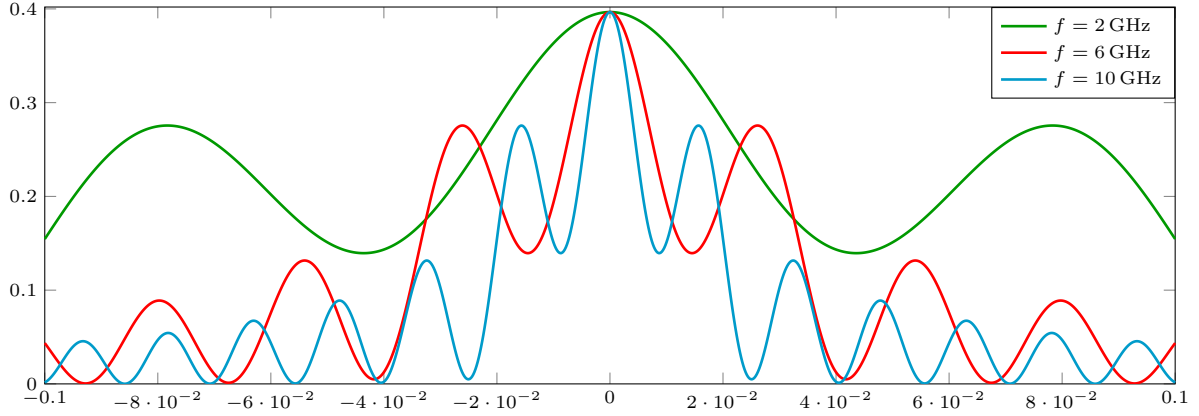
\begin{figure}[h]
\begin{center}
\begin{tikzpicture}
\scriptsize
\begin{axis}
[legend style={at={(1,1)}},
width=\textwidth,
height=0.4\textwidth,
xmin=-0.1,
xmax=0.1,
ymin=0,
ymax=0.402,
legend cell align={left}]
\addplot[line width=1pt,solid,color=green!60!black] %
	table[x=x,y=y1,col sep=comma]{PlotMSM1.csv};
\addlegendentry{\scriptsize$f=\SI{2}{\giga\hertz}$};
\addplot[line width=1pt,solid,color=red] %
	table[x=x,y=y2,col sep=comma]{PlotMSM1.csv};
\addlegendentry{\scriptsize$f=\SI{6}{\giga\hertz}$};
\addplot[line width=1pt,solid,color=cyan!80!black] %
	table[x=x,y=y3,col sep=comma]{PlotMSM1.csv};
\addlegendentry{\scriptsize$f=\SI{10}{\giga\hertz}$};
\end{axis}
\end{tikzpicture}
\caption{\label{PlotMSM1}Plots of $|\mathcal{D}_{\msm}^{(1)}(x)|$ for $-0.1\leq x\leq0.1$ at $f=2,6,\SI{10}{\giga\hertz}$.}
\end{center}
\end{figure}

\begin{remark}[Influence of the factor $\mathcal{E}_{\msm}(\mr,m)$]\label{RemarkM2}
Similar to the $\mathcal{E}_{\osm}(\mr,m)$ in \eqref{Structure_Single}, the factor $\mathcal{E}_{\msm}(\mr)$ does not contribute to the identification of objects, and it disturbs the identification process by generating several artifacts. Thus, to examine the influence of $\mathcal{E}_{\msm}(\mr)$, we consider the following quantity:
\begin{align*}\mathcal{D}_{\msm}^{(2)}(x)=&\frac{3}{2\pi}\sum_{q=1}^{10^5}\frac{(-1)^q}{q}\sin\frac{2q\pi}{3}J_p(k|x|)^2\\
&+\frac{3}{2\pi}\sum_{q=1}^{10^5}\frac{(-1)^{q+1}}{q+2}\sin\frac{2(q+2)\pi}{3}J_q(k|x|)^2+\frac{3}{2\pi}\sum_{q=1,q\ne2}^{10^5}\frac{(-1)^{q-1}}{q-2}\sin\frac{2(q-2)\pi}{3}J_q(k|x|)^2.
\end{align*}
Based on the numerical computation, we can examine that $|\mathcal{D}_{\msm}^{(2)}(x)|\leq0.1$, refer to Figure \ref{PlotMSM2}. Thus, the factor $\mathcal{E}_{\msm}(\mr)$ will contribute to the generation of some artifacts with small magnitudes; however,  its value can be disregarded such that it will not disturb the recognition of the existence of objects.
\end{remark}

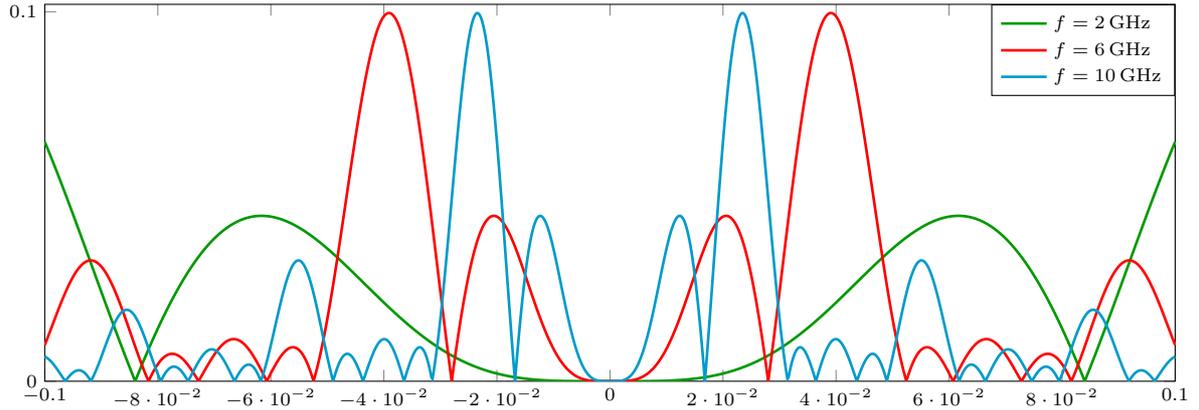
\begin{figure}[h]
\begin{center}
\begin{tikzpicture}
\scriptsize
\begin{axis}
[legend style={at={(1,1)}},
width=\textwidth,
height=0.4\textwidth,
ytick distance=0.1,
xmin=-0.1,
xmax=0.1,
ymin=0,
ymax=0.102,
legend cell align={left}]
\addplot[line width=1pt,solid,color=green!60!black] %
	table[x=x,y=y1,col sep=comma]{PlotMSM2.csv};
\addlegendentry{\scriptsize$f=\SI{2}{\giga\hertz}$};
\addplot[line width=1pt,solid,color=red] %
	table[x=x,y=y2,col sep=comma]{PlotMSM2.csv};
\addlegendentry{\scriptsize$f=\SI{6}{\giga\hertz}$};
\addplot[line width=1pt,solid,color=cyan!80!black] %
	table[x=x,y=y3,col sep=comma]{PlotMSM2.csv};
\addlegendentry{\scriptsize$f=\SI{10}{\giga\hertz}$};
\end{axis}
\end{tikzpicture}
\caption{\label{PlotMSM2}Plots of $|\mathcal{D}_{\msm}^{(2)}(x)|$ for $-0.1\leq x\leq0.1$ at $f=2,6,\SI{10}{\giga\hertz}$.}
\end{center}
\end{figure}

\begin{remark}[Comparing imaging performance with a single source OSM]\label{RemarkM3}
Based on Corollary \ref{Corollary_Comparison}, since $\mathcal{E}_{\msm}(\mr)$ has a smaller amplitude than $\mathcal{E}_{\osm}(\mr,m)$, the factor $\mathcal{E}_{\msm}(\mr)$ will generate fewer artifacts than $\mathcal{E}_{\osm}(\mr,m)$. Therefore, the imaging performance of $\mathfrak{F}_{\msm}(\mr)$ will be better than the one of $\mathfrak{F}_{\osm}(\mr,m)$ for all $m=1,2,\ldots,M$.
\end{remark}

\begin{remark}[Unique determination of objects]
Following Remarks \ref{RemarkM1} and \ref{RemarkM2}, the value of $\mathfrak{F}_{\msm}(\mr)$ will reach its local maxima when $\mr=\mr_s\in D_s$, $s=1,2,\ldots,S$. For visualization, we consider the following quantity:
\[\mathcal{D}_{\msm}(x)=\mathcal{D}_{\msm}^{(1)}(x)+\mathcal{D}_{\msm}^{(2)}(x),\]
where $\mathcal{D}_{\msm}^{(1)}(x)$ and $\mathcal{D}_{\msm}^{(2)}(x)$ are given in Remarks \ref{RemarkM1} and \ref{RemarkM2}, respectively. This is similar to the value of $\mathfrak{F}_{\msm}(\mr)$ in the presence of single object located at the origin. Based on the plots of $\mathcal{D}_{\msm}(x)$ in Figure \ref{PlotMSM}, we can examine the the value of $\mathcal{F}_{\msm}(x)$ reaches its maximum value at $x=0$. Therefore, we can numerically verify that the objects $D_s$ can be identified uniquely through the map of $\mathfrak{F}_{\msm}(\mr)$.
\end{remark}

\begin{figure}[h]
\begin{center}
\begin{tikzpicture}
\scriptsize
\begin{axis}
[legend style={at={(1,1)}},
width=\textwidth,
height=0.4\textwidth,
xmin=-0.1,
xmax=0.1,
ymin=0,
ymax=0.402,
legend cell align={left}]
\addplot[line width=1pt,solid,color=green!60!black] %
	table[x=x,y=y1,col sep=comma]{PlotMSM.csv};
\addlegendentry{\scriptsize$f=\SI{2}{\giga\hertz}$};
\addplot[line width=1pt,solid,color=red] %
	table[x=x,y=y2,col sep=comma]{PlotMSM.csv};
\addlegendentry{\scriptsize$f=\SI{6}{\giga\hertz}$};
\addplot[line width=1pt,solid,color=cyan!80!black] %
	table[x=x,y=y3,col sep=comma]{PlotMSM.csv};
\addlegendentry{\scriptsize$f=\SI{10}{\giga\hertz}$};
\end{axis}
\end{tikzpicture}
\caption{\label{PlotMSM}Plots of $|\mathcal{D}_{\msm}(x)|$ for $-0.1\leq x\leq0.1$ at $f=2,6,\SI{10}{\giga\hertz}$.}
\end{center}
\end{figure}
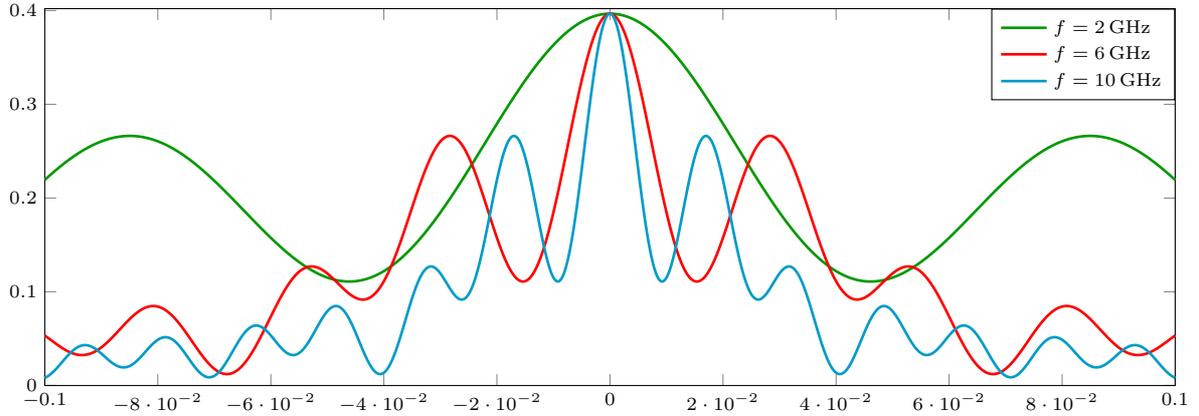

\section{Simulation results with multiple sources using synthetic and experimental dataset}\label{sec:6}
Here, we present the corresponding numerical simulation results using synthetic and experimental dataset. The simulation configuration is the same as that described in Section \ref{sec:4}, except that the range of the emitters is $\SI{0}{\degree}$ to $\SI{350}{\degree}$ with a step size of $\SI{10}{\degree}$. 

\begin{example}[Results using synthetic dataset in Cases $1$ and $2$]\label{Ex_Synthetic_Multiple1}
Figure \ref{Synthetic_Multiple_Case1} shows maps of $\mathfrak{F}_{\msm}(\mr)$ in Case $1$. By comparing with the results in Example \ref{Ex_Single_Case1}, it is possible to recognize the existence of $D_s$, and their outline shapes were successfully retrieved. Similarly, although values of $\mathfrak{F}_{\msm}(\mr)$ on $D_s$ are different, the existence and shapes of $D_s$ in Case $2$ are discernible, refer to Figure \ref{Synthetic_Multiple_Case2}.
\end{example}

\begin{figure}[h]
\begin{center}
\subfigure[$f=\SI{4}{\giga\hertz}$]{\includegraphics[width=.33\columnwidth]{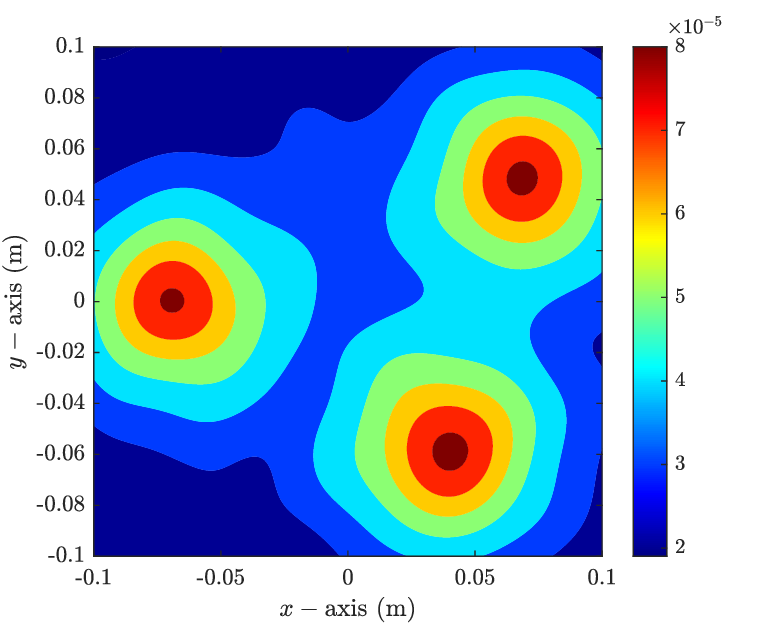}}\hfill
\subfigure[$f=\SI{8}{\giga\hertz}$]{\includegraphics[width=.33\columnwidth]{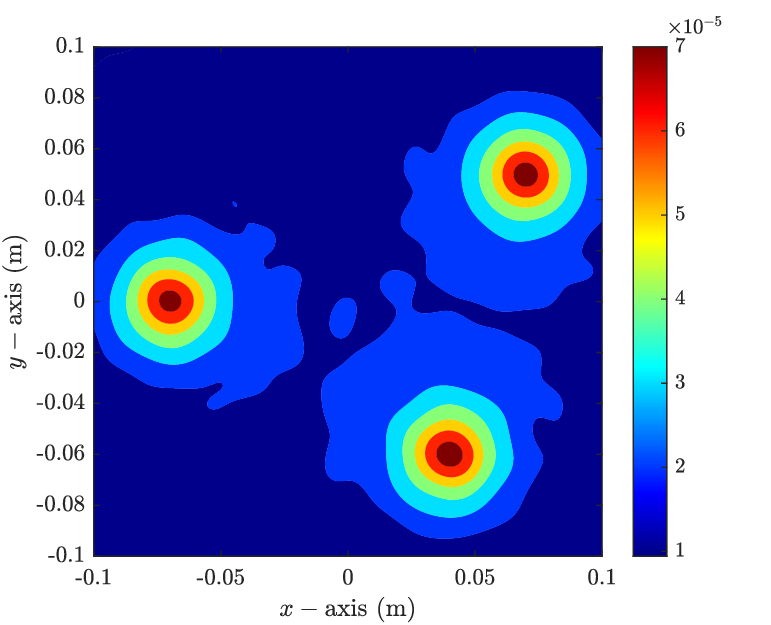}}\hfill
\subfigure[$f=\SI{12}{\giga\hertz}$]{\includegraphics[width=.33\columnwidth]{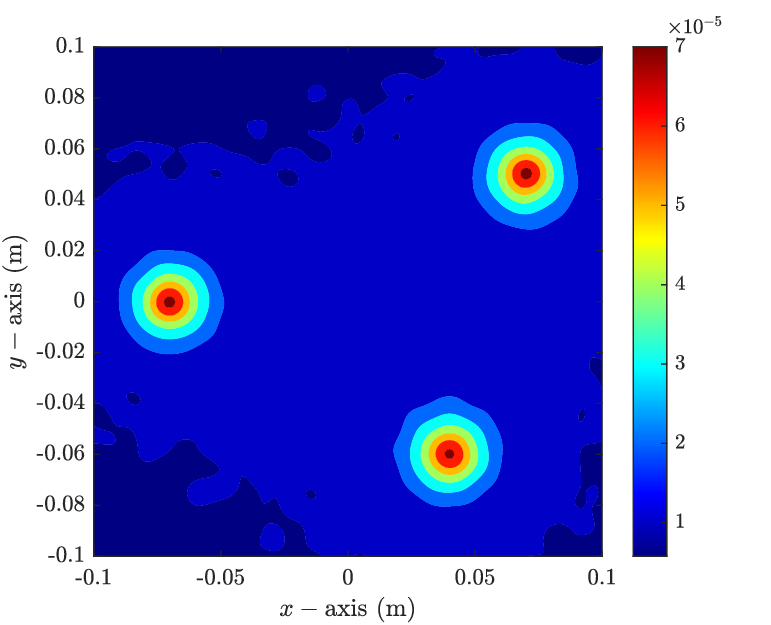}}
\caption{\label{Synthetic_Multiple_Case1}(Example \ref{Ex_Synthetic_Multiple1}) Maps of $\mathfrak{F}_{\msm}(\mr)$ in Case $1$.}
\end{center}
\end{figure}

\begin{figure}[h]
\begin{center}
\subfigure[$f=\SI{4}{\giga\hertz}$]{\includegraphics[width=.33\columnwidth]{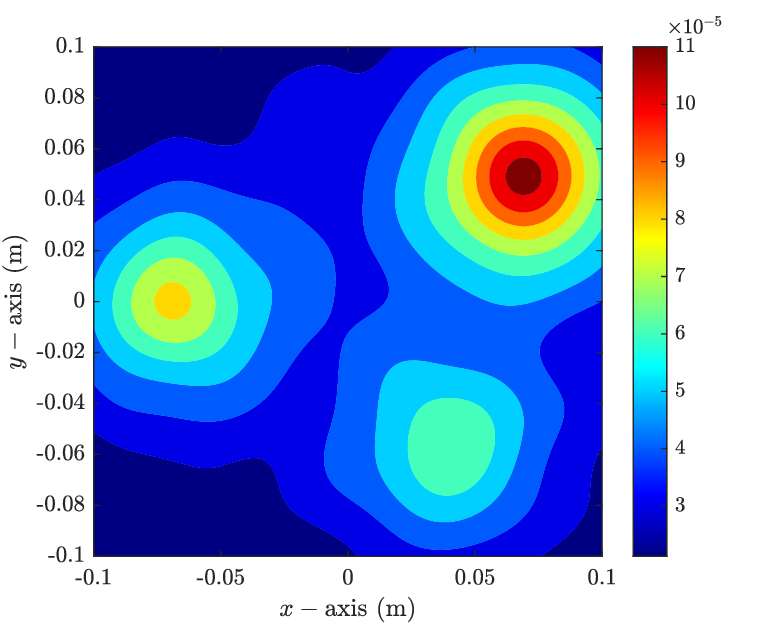}}\hfill
\subfigure[$f=\SI{8}{\giga\hertz}$]{\includegraphics[width=.33\columnwidth]{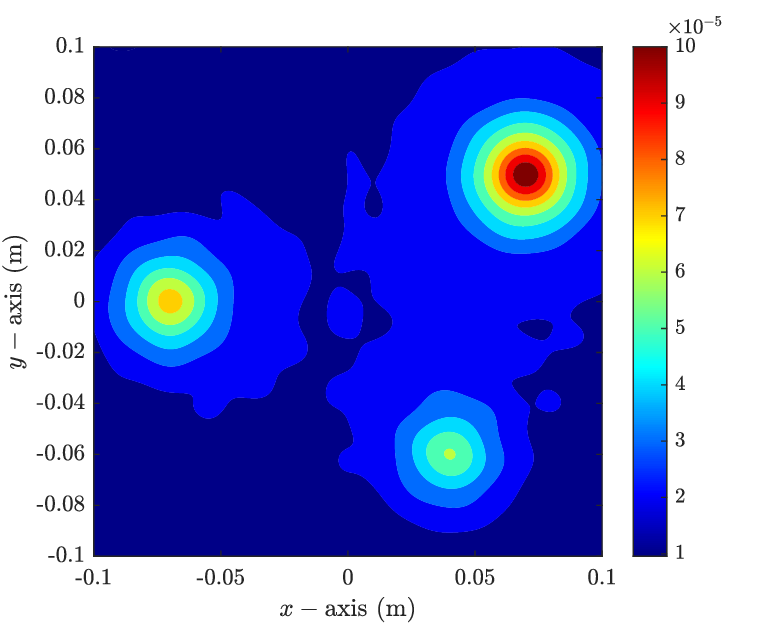}}\hfill
\subfigure[$f=\SI{12}{\giga\hertz}$]{\includegraphics[width=.33\columnwidth]{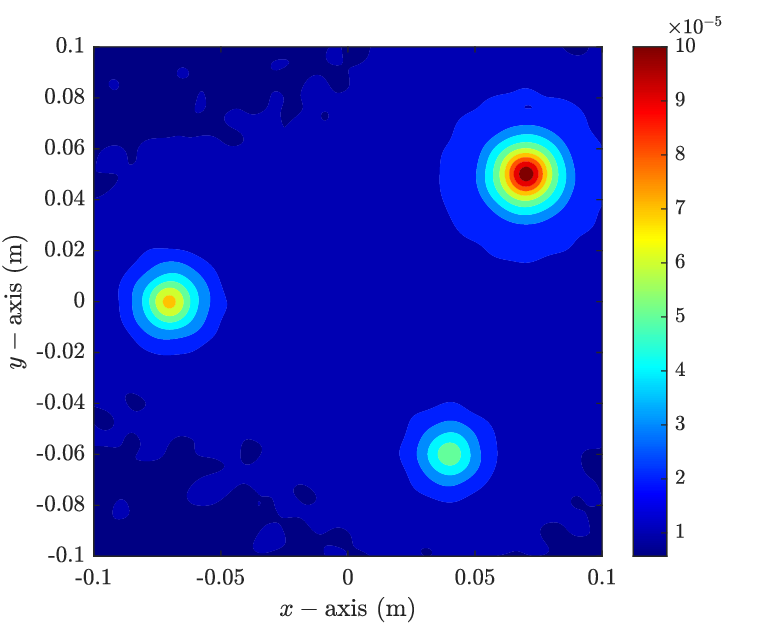}}
\caption{\label{Synthetic_Multiple_Case2}(Example \ref{Ex_Synthetic_Multiple1}) Maps of $\mathfrak{F}_{\msm}(\mr)$ in Case $2$.}
\end{center}
\end{figure}

\begin{example}[Results using synthetic dataset in Cases $3$ and $4$]\label{Ex_Synthetic_Multiple2}
Figure \ref{Synthetic_Multiple_Case3} shows the maps of $\mathfrak{F}_{\msm}(\mr)$ in Case $3$. In contrast to the Case $2$, the existence of $D_3$ cannot be recognized through the maps at $f=\SI{4}{\giga\hertz}$ and $f=\SI{8}{\giga\hertz}$. It is interesting to note that a circle with a very small magnitude appears in the neighborhood of $D_3$, but it is challenging to distinguish it from the artifacts. 

Figure \ref{Synthetic_Multiple_Case4} shows the maps of $\mathfrak{F}_{\msm}(\mr)$ in Case $4$. Similar phenomena to those in Figure \ref{Synthetic_Multiple_Case3} ($D_3$ cannot be recognized) can be observed, but since the value of $\mathfrak{F}_{\msm}(\mr)$ at $\mr_2$ is too small, it is very difficult to distinguish whether it is $D_2$ or an artifact.
\end{example}

\begin{figure}[h]
\begin{center}
\subfigure[$f=\SI{4}{\giga\hertz}$]{\includegraphics[width=.33\columnwidth]{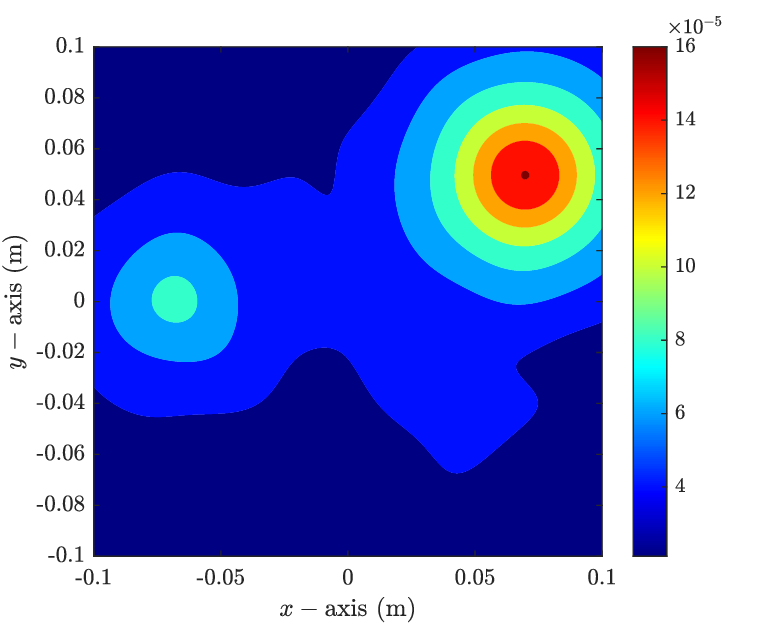}}\hfill
\subfigure[$f=\SI{8}{\giga\hertz}$]{\includegraphics[width=.33\columnwidth]{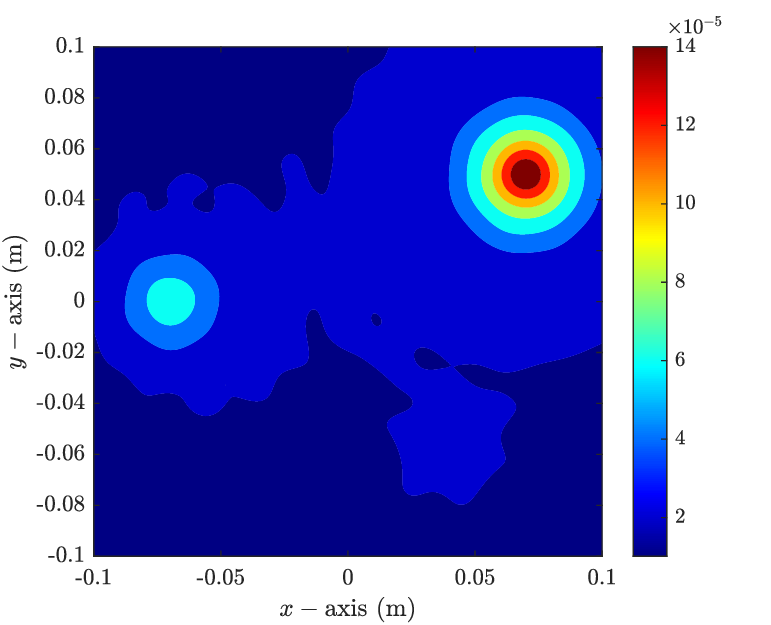}}\hfill
\subfigure[$f=\SI{12}{\giga\hertz}$]{\includegraphics[width=.33\columnwidth]{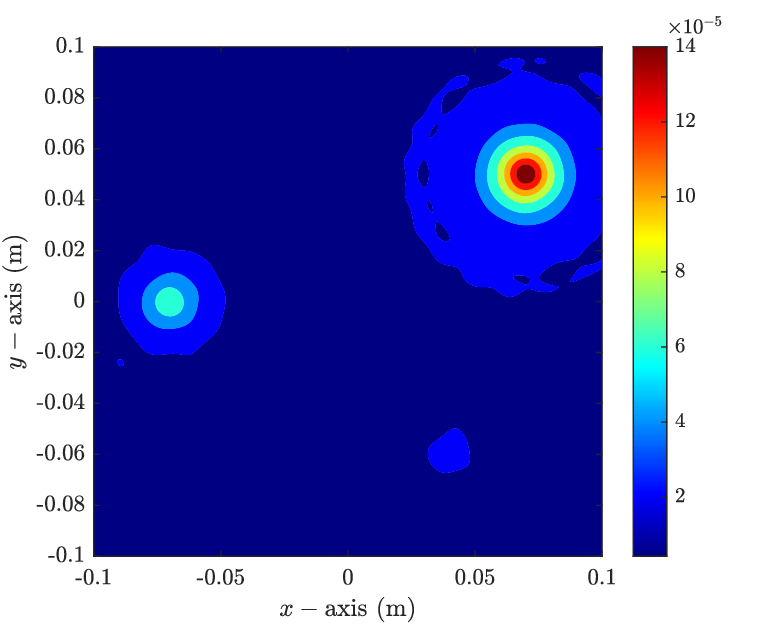}}
\caption{\label{Synthetic_Multiple_Case3}(Example \ref{Ex_Synthetic_Multiple2}) Maps of $\mathfrak{F}_{\msm}(\mr)$ in Case $3$.}
\end{center}
\end{figure}

\begin{figure}[h]
\begin{center}
\subfigure[$f=\SI{4}{\giga\hertz}$]{\includegraphics[width=.33\columnwidth]{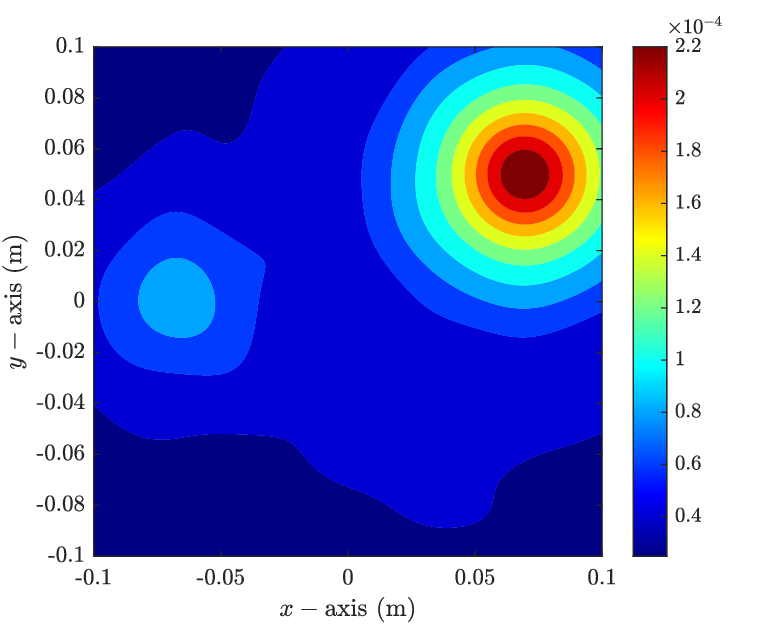}}\hfill
\subfigure[$f=\SI{8}{\giga\hertz}$]{\includegraphics[width=.33\columnwidth]{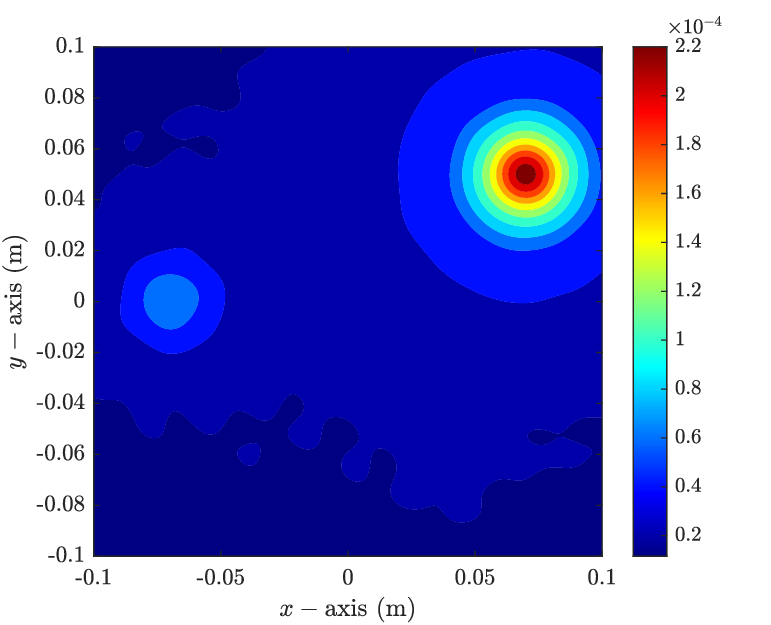}}\hfill
\subfigure[$f=\SI{12}{\giga\hertz}$]{\includegraphics[width=.33\columnwidth]{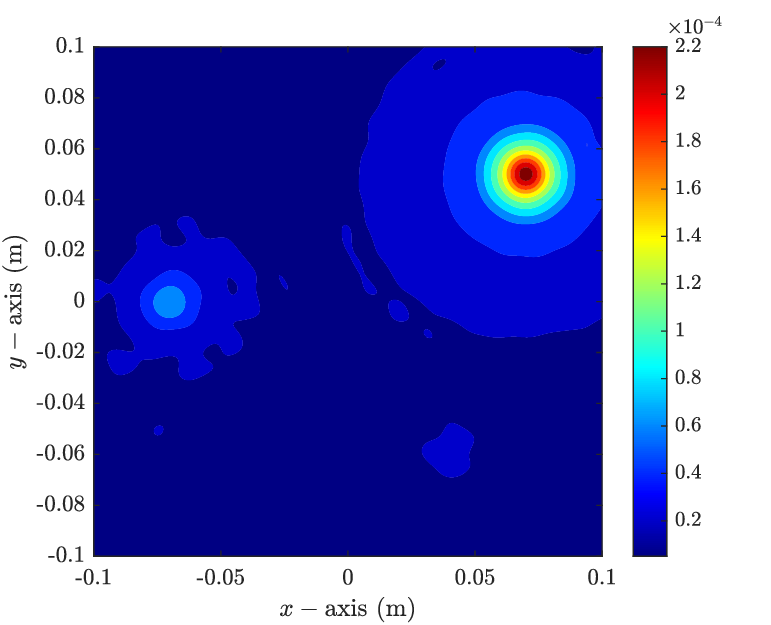}}
\caption{\label{Synthetic_Multiple_Case4}(Example \ref{Ex_Synthetic_Multiple2}) Maps of $\mathfrak{F}_{\msm}(\mr)$ in Case $4$.}
\end{center}
\end{figure}

\begin{example}[Results using Fresnel experimental dataset]\label{Ex_Fresnel_Multiple}
Figure \ref{Fresnel_Multiple} shows the maps of $\mathfrak{F}_{\msm}(\mr)$ with various frequencies. In contrast to the imaging results obtained with a single source, here, the existence of the object can be recognized with any frequency setting. However, the shape of the object cannot be identified if a low frequency is applied, e.g. $f=2,4,\SI{8}{\giga\hertz}$. Fortunately, the location and outline shape of the object can be retrieved successfully when a sufficiently high frequency is applied, i.e., $f=10,12,\SI{16}{\giga\hertz}$.
\end{example}

\begin{figure}[h]
\begin{center}
\subfigure[$f=\SI{2}{\giga\hertz}$]{\includegraphics[width=.33\columnwidth]{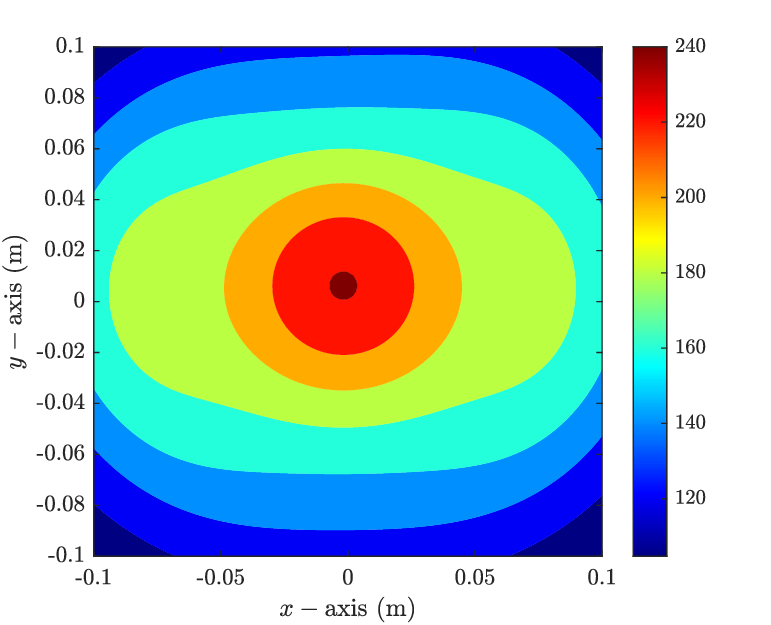}}\hfill
\subfigure[$f=\SI{4}{\giga\hertz}$]{\includegraphics[width=.33\columnwidth]{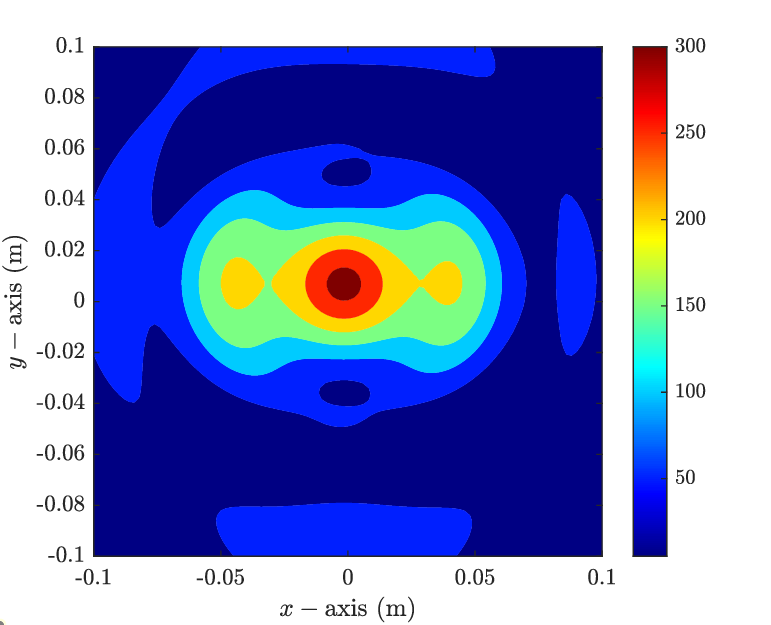}}\hfill
\subfigure[$f=\SI{8}{\giga\hertz}$]{\includegraphics[width=.33\columnwidth]{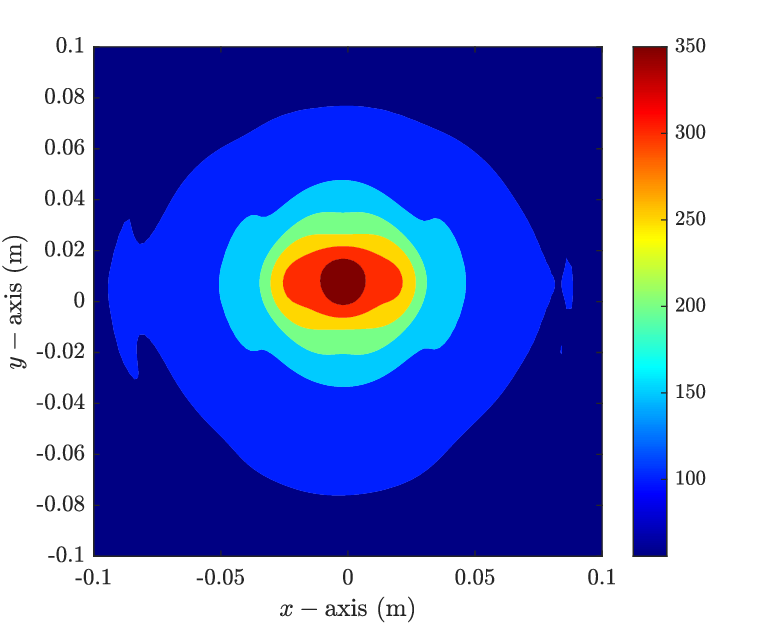}}\\
\subfigure[$f=\SI{10}{\giga\hertz}$]{\includegraphics[width=.33\columnwidth]{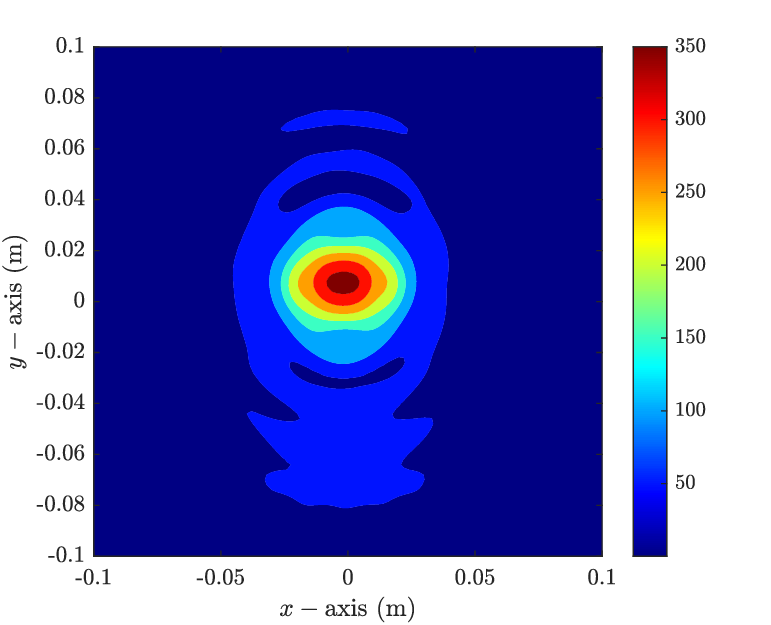}}\hfill
\subfigure[$f=\SI{12}{\giga\hertz}$]{\includegraphics[width=.33\columnwidth]{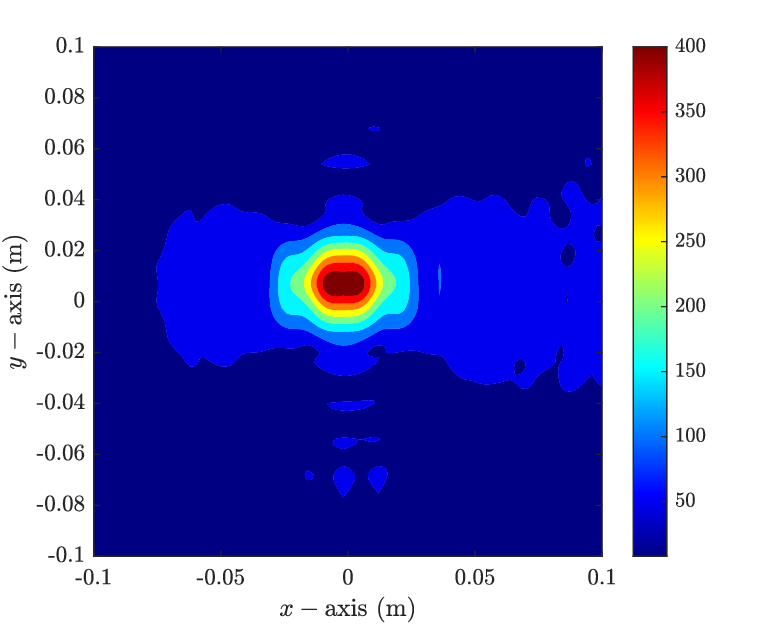}}\hfill
\subfigure[$f=\SI{16}{\giga\hertz}$]{\includegraphics[width=.33\columnwidth]{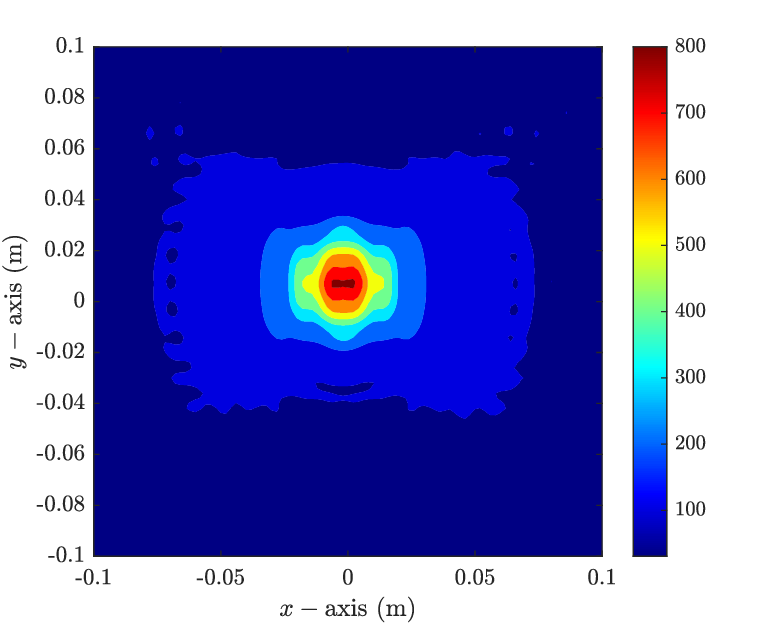}}
\caption{\label{Fresnel_Multiple} (Example \ref{Ex_Fresnel_Multiple}) Maps of $\mathfrak{F}_{\msm}(\mr)$.}
\end{center}
\end{figure}

\begin{example}[Further results using Fresnel experimental dataset]\label{Ex_Fresnel_MultipleFrequencies}
For the final example, we consider imaging with multiple sources and frequencies. Here, we define a set $F=\set{f:f=2,4,6,8,10,12,14,\SI{16}{\giga\hertz}}$, denote $\mathfrak{F}_{\msm}(\mr,f)$ as the imaging function $\mathfrak{F}_{\msm}(\mr)$ with the frequency of operation $f$, and introduce the following three imaging functions:
\[\mathfrak{F}_{\fsm}^{(1)}(\mr)=\bigg|\sum_{f\in F}\frac{\mathfrak{F}_{\msm}(\mr,f)}{\displaystyle\max_{\mr\in\Omega}|\mathfrak{F}_{\msm}(\mr,f)|}\bigg|,\quad\mathfrak{F}_{\fsm}^{(2)}(\mr)=\bigg|\sum_{f\in F}\mathfrak{F}_{\msm}(\mr,f)\bigg|,\quad\mathfrak{F}_{\fsm}^{(3)}(\mr)=\sum_{f\in F}|\mathfrak{F}_{\msm}(\mr,f)|.\]
\end{example}

Figure \ref{Fresnel_MultipleFrequencies} shows the maps of $\mathfrak{F}_{\fsm}^{(j)}(\mr)$, $j=1,2,3$. Note that the outline shape of the object can be recognized through only the map of $\mathfrak{F}_{\fsm}^{(1)}(\mr)$; however, its existence and location can be identified through the maps $\mathfrak{F}_{\fsm}^{(j)}(\mr)$, $j=1,2,3$.

\begin{figure}[h]
\begin{center}
\subfigure[Map of $\mathfrak{F}_{\fsm}^{(1)}(\mr)$]{\includegraphics[width=.33\columnwidth]{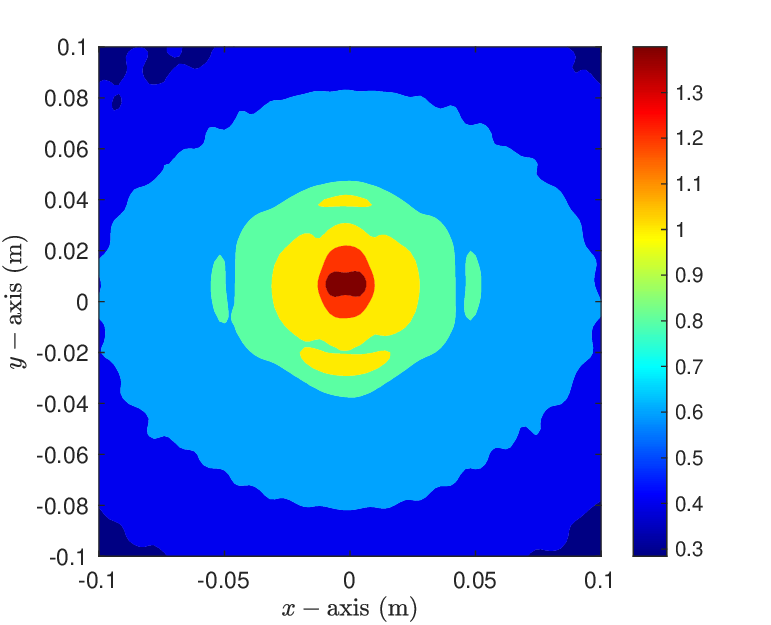}}\hfill
\subfigure[Map of $\mathfrak{F}_{\fsm}^{(2)}(\mr)$]{\includegraphics[width=.33\columnwidth]{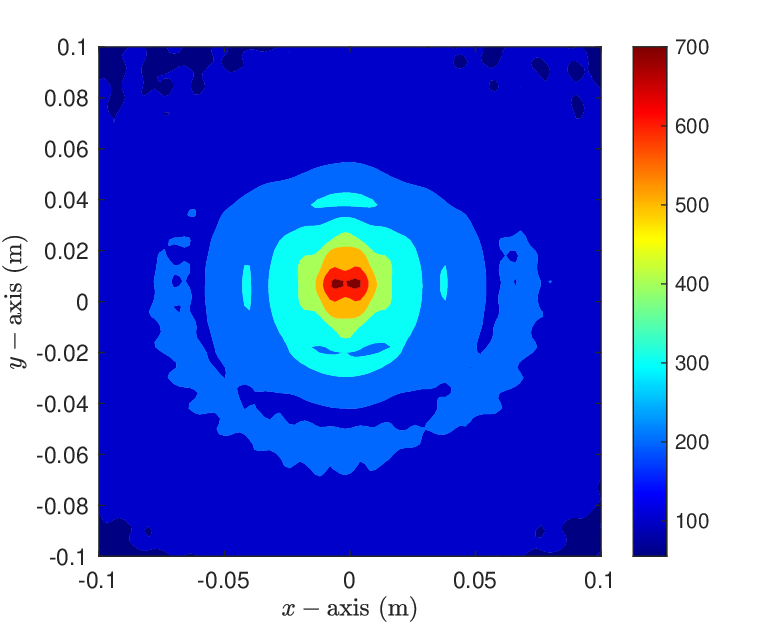}}\hfill
\subfigure[Map of $\mathfrak{F}_{\fsm}^{(3)}(\mr)$]{\includegraphics[width=.33\columnwidth]{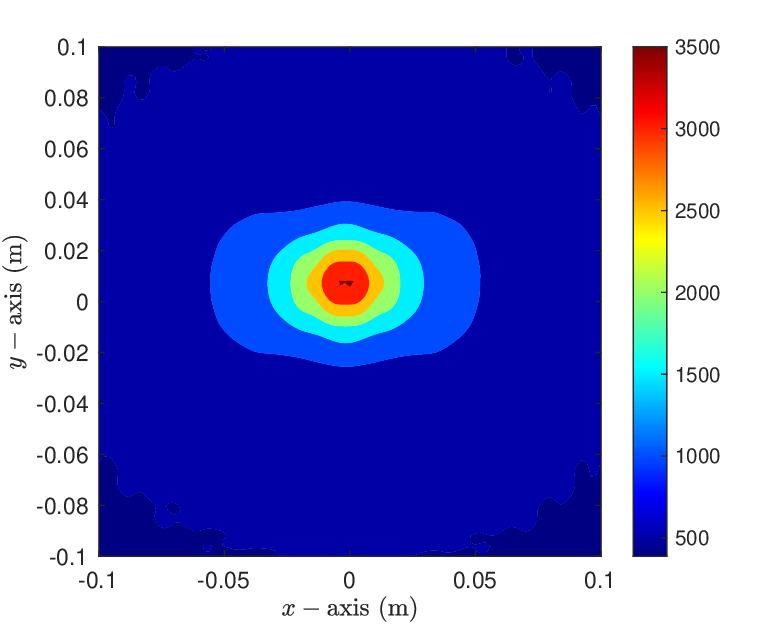}}
\caption{\label{Fresnel_MultipleFrequencies}(Example \ref{Ex_Fresnel_MultipleFrequencies}) Maps of $\mathfrak{F}_{\fsm}^{(j)}(\mr)$, $j=1,2,3$.}
\end{center}
\end{figure}

\section{Conclusion}\label{sec:7}
This paper has reported on the application of the OSM with a single source to identify the existence and outline shape of a small object in the TE polarization process from the Fresnel experimental dataset. In this study, based on the asymptotic expansion formula for the scattered field data in the presence of small objects, we designed an indicator function of the OSM, and we derived an accurate relationship between the indicator function and an infinite series of the Bessel function of the first kind.
This paper has discussed both the applicability and limitations of the OSM, and the results of numerical simulation conducted at various frequencies with the 2D Fresnel experimental dataset have been reported to verify the theoretical result.

Then, to improve the imaging performance of the OSM, we applied multiple sources and introduced an imaging function. Through careful analysis, we have demonstrated that the proposed indicator function can be expressed by an accurate relationship between the indicator function and an infinite series of the Bessel function of the first kind, and we have demonstrated that the proposed indicator function improves the imaging performance. Based on the theoretical result, we numerically demonstrated that it is possible to identify small objects uniquely.

In this study, we considered the application of the OSM in TE polarization from the 2D Fresnel experimental dataset. In the future, we plan to extend our work to the 3D Fresnel experimental dataset \cite{GSE}, and monostatic and bistatic measurement configurations \cite{KLP3,KLP4}.

\section*{Acknowledgments}
This research was supported by the National Research Foundation of Korea (NRF) grant funded by the Korea government (MSIT) (NRF-RS-2020-NR048569).

\end{document}